\documentclass[12pt,reqno]{amsart}

\usepackage{amsmath,amssymb,amsthm}
\usepackage{mathrsfs}
\usepackage[utf8]{inputenc}
\usepackage[T1]{fontenc}
\usepackage{appendix}
\usepackage{enumitem}
\usepackage[protrusion=true,expansion=false]{microtype}
\usepackage{tikz-cd}
\usepackage[margin=1in]{geometry}
\usepackage{mathtools}

 \usepackage[colorlinks=true, linkcolor=blue, citecolor=blue]{hyperref}      
\usepackage{orcidlink}

\usepackage[capitalize,nameinlink]{cleveref}
\usepackage{todonotes}

\theoremstyle{plain}
\newtheorem{theorem}{Theorem}[section]
\newtheorem{proposition}[theorem]{Proposition}
\newtheorem{lemma}[theorem]{Lemma}
\newtheorem{corollary}[theorem]{Corollary}
\newtheorem*{acknow}{\textup{Acknowledgement}}

\theoremstyle{definition}
\newtheorem{definition}[theorem]{Definition}

\theoremstyle{remark}
\newtheorem{remark}[theorem]{Remark}

\newcommand{\OP}{\mathbb{OP}}

\newcommand{\HP}{\mathbb{HP}}

\newcommand{\Z}{\mathbb{Z}}
\newcommand{\Q}{\mathbb{Q}}
\newcommand{\R}{\mathbb{R}}
\newcommand{\s}{\mathbb{S}}

\newcommand{\diff}{\mathrm{DIFF}}

\newcommand{\Sdiff}{\mathcal{S}^{\diff}}

\newcommand{\Ndiff}{\mathcal{N}^{\diff}}

\newcommand{\sh}{\operatorname{sh}}
\newcommand{\ph}{\operatorname{ph}}
\newcommand{\ch}{\operatorname{ch}}
\newcommand{\AffiliationsBlock}{%
  \vspace{0.8em}
  \begingroup
  \centering
  \footnotesize\itshape
  $^{a}$ Department of Mathematics, Indian Institute of Technology Madras, \\
  Chennai-600036, Tamil Nadu, India.\\\vspace{2mm}
  $^{b}$ The Institute of Mathematical Sciences, A CI of Homi Bhabha National Institute,\\
  Chennai-600113,  Tamil Nadu, India.\par
  \endgroup
  \vspace{-1.3em}
}\makeatletter
\patchcmd{\@maketitle}{%
  \ifx\@empty\@dedicatory
  \else
}{%
  \AffiliationsBlock
  \vspace{1.5em}
  \ifx\@empty\@dedicatory
  \else
}{}{\message{Patch failed!}}
\makeatother

\begin{document}
\title[Structure sets of thickenings of the Cayley plane]{Relative
smooth surgery structure sets of thickenings of the Cayley
projective plane and applications}
 \author[S. Mandal]{{Souvik Mandal$^{\,a,*}$}
 }
 \author[A. Sarkar]{{Ankur Sarkar$^{\,b}$}
 }
\thanks{\hspace{-1.1em}* Corresponding author.}
\thanks{\raggedright\hspace{0.53em}
  \makebox[8.5em][l]{\textit{E-mail addresses:}}%
  \begin{minipage}[t]{0.75\textwidth}
    \texttt{ma22d014@smail.iitm.ac.in}, \texttt{ssouvik.xyz@gmail.com} (S. Mandal); \\
    \texttt{ankurimsc@gmail.com} (A. Sarkar).
  \end{minipage}}

\begin{abstract}
We compute the relative smooth surgery structure sets of the
thickenings $\mathbb{OP}^{2}\times\mathbb{D}^{k}$ of the Cayley
projective plane $\mathbb{OP}^{2}$ for every $k\geq 1$ with
$k\equiv 0\pmod 4$, by determining the corresponding normal
invariants and surgery obstruction map. We show that the latter is
not surjective and determine the $2$-adic valuation of the generator
of its image. As applications, we construct infinitely many
pairwise non-homeomorphic closed smooth manifolds of dimension
$16+k$, homotopy equivalent to $\mathbb{OP}^{2}\times\mathbb{S}^{k}$
and distinguished by their Pontryagin numbers; we compute the rational homotopy
groups of the block diffeomorphism group
$\widetilde{\operatorname{Diff}}(\mathbb{OP}^{2})$ in every degree congruent to $3$ modulo $4$; and we construct smooth
$\mathbb{OP}^{2}$-bundles over $\mathbb{S}^{4}$, $\mathbb{S}^{8}$,
and $\mathbb{S}^{12}$ whose total spaces have non-vanishing
$\widehat{\mathfrak{A}}$-genus. These bundles yield elements of infinite order in the homotopy
groups of the spaces of metrics of
positive sectional, Ricci, and scalar curvature on
$\mathbb{OP}^{2}$ in degrees $3$, $7$, and $11$.
\end{abstract}

\maketitle
\vspace{-0.5em}
\begin{center}
\begin{minipage}{0.845\textwidth}
    \footnotesize
    \begin{list}{}{%
        \leftmargin=5.5em 
        \labelwidth=5.5em
        \labelsep=0pt \parsep=0pt \topsep=0pt \itemsep=0pt
    }
        \item[\textit{Keywords:}\hfill] Cayley Projective Plane, Higher Smooth Surgery Structure
Sets, Normal Invariants, Adams Operations, Bernoulli Numbers, Block
Diffeomorphism Groups, Positive Sectional Curvature.
    \end{list}

    \vspace{4pt} 

    \begin{list}{}{%
        \leftmargin=18.5em 
        \labelwidth=18.5em
        \labelsep=0pt \parsep=0pt \topsep=0pt \itemsep=0pt
    }
        \item[2020\hspace{1mm}\textit{Mathematics Subject Classification:}\hfill] {Primary 57R65, 57R67;\\ Secondary 57R55, 55N15,
55Q50, 57S05, 58D17}
    \end{list}
\end{minipage}
\end{center}
\vspace{1em}
\section{Introduction}
In surgery theory, the relative smooth structure sets play a central role in the classification of smooth manifolds with boundary within a fixed homotopy type. These sets have been studied for a handful of manifolds,
beginning with the foundational treatment of Wall~\cite{Wal99} and, more
recently, in the work of Chang \emph{et al.} \cite{stanly}. Kalu\v{z}n\'y and Macko \cite{KM26} studied the relative smooth structure
sets of $X=\mathbb{CP}^{n}\times\mathbb{D}^{k}$, the thickenings of the
complex projective spaces. 
In this article, we study the relative smooth surgery structure sets of the thickenings of the Cayley projective plane. The Cayley projective plane $\OP^{2}$ (or octonionic projective plane) is a closed simply connected $16$-dimensional smooth manifold with CW descomposition $\OP^2=\mathbb{S}^{8}\cup_{\sigma}e^{16}$, where $\sigma\in\pi_{15}(\mathbb{S}^{8})$ is the octonionic Hopf map of Hopf invariant one \cite{Lac21}. 
It is the exceptional compact rank one symmetric space, and in particular
admits a metric of positive sectional curvature \cite[pp.~208--210]{peterson}.
Unlike the real, complex, and quaternionic cases, the non-associativity of
$\mathbb{O}$ implies that there are no higher-dimensional octonionic
projective spaces (see \cite{Lac21}), and hence $\OP^{2}$ is the only
octonionic projective plane.

For a compact smooth manifold $X$ with boundary, let $\Sdiff_{\partial}(X)$
denote its relative smooth structure set, and let $\mathcal{E}_{\partial}(X)$
denote the group of homotopy classes, relative to $\partial X$, of
self-homotopy equivalences of $(X,\partial X)$ restricting to the identity on
$\partial X$. This
group acts on $\Sdiff_{\partial}(X)$ by post-composition, and the resulting
orbit space $\Sdiff_{\partial}(X)\big/\mathcal{E}_{\partial}(X)$ is in
one-to-one correspondence with the set of diffeomorphism classes of compact
smooth manifolds $M$ admitting a homotopy equivalence of pairs
$(M,\partial M)\to(X,\partial X)$ restricting to a diffeomorphism on the
boundary \cite{KT01}. Thus the relative smooth structure set and the action of
$\mathcal{E}_{\partial}(X)$ on it are the two fundamental ingredients in this
classification problem.


For $X=\OP^{2}\times\mathbb{D}^{k}$ the group
$\mathcal{E}_{\partial}(\OP^{2}\times\mathbb{D}^{k})$ is the more
tractable of the two, containing at most four elements.
Indeed, forgetting the boundary condition we can view it inside the group $\mathcal{E}(\OP^{2}\times\mathbb{D}^{k})$ of homotopy
classes of self-homotopy equivalences of $\OP^2\times \mathbb{D}^k$.
Furthermore, the deformation retraction of $\OP^{2} \times \mathbb{D}^{k}$ onto $\OP^{2}$ identifies $\mathcal{E}(\OP^{2} \times \mathbb{D}^{k})$ with $\mathcal{E}(\OP^{2})$, which is studied in \cite{Oka}.
In this article, we compute the relative smooth structure set
$\Sdiff_{\partial}(\OP^{2}\times\mathbb{D}^{k})$ for
$k\equiv0\pmod4$ and $k\ge4$ by analyzing the associated relative
smooth surgery exact sequence.


The case $k=0$ was been studied by Eells--Kuiper
\cite{EK62a, EK62b}, Kramer \cite{Kra03}, Kramer--Stolz \cite{KS07}
and Su--Yang \cite{SW}. In this case, $\Sdiff_{\partial}(\OP^{2})$ is merely a pointed set, while for
$k\geq1$ the stacking operation endows
$\Sdiff_{\partial}(\OP^{2}\times\mathbb{D}^{k})$ with a group structure,
which is abelian for $k\geq2$ \cite[\S11.8]{LM24}.

A further motivation for computing the relative smooth structure sets
$\Sdiff_{\partial}(X\times\mathbb{D}^{k})$ is their close
relationship with the homotopy groups of certain automorphism spaces
of $X$.
In fact, if $X$ is a closed oriented smooth manifold of dimension at least five, then for every $k\geq 1$ there is an isomorphism \cite[1.1.4]{WW01} (see also \cite[p.~33]{stability}) 
\[
\Sdiff_{\partial}(X\times\mathbb{D}^{k})
\;\cong\;
\pi_{k}\bigl(\mathrm{hAut}(X)/
\widetilde{\operatorname{Diff}}(X)\bigr),
\qquad k\geq 1,
\]
where $\mathrm{hAut}(X)$ denotes the topological monoid of
orientation-preserving self-homotopy equivalences of $X$, and
$\widetilde{\operatorname{Diff}}(X)$ the block diffeomorphism group of $X$. The applications of this isomorphism to $X=\OP^{2}$ are discussed in Section~\ref{sec:applications}. 

We now state the main results of the article.

\begin{theorem}
\label{thm:main-structure}
For every $k \geq 1$ with $k \equiv 0 \pmod{4}$, there is an isomorphism of
abelian groups
\[
   \Sdiff_\partial(\OP^2 \times \mathbb{D}^k)
   \;\cong\;
   \Z^{2}\;\oplus\;T_k\;\oplus\;\Theta_k ,
\]
where $\Theta_k$ denotes the Kervaire-Milnor group
\cite{homotopysphere} and $T_{k}$ is the cokernel of the homomorphism
$[\Sigma^{k}\OP^{2},\Omega J]\colon
\bigl[\Sigma^{k}\OP^{2},\,O\bigr]\rightarrow\bigl[\Sigma^{k}\OP^{2},\,G\bigr]$
induced by the canonical map $J\colon BO\rightarrow BG$. Explicitly,
\[
   T_k \;=\;
   \begin{cases}
      \bigl[\Sigma^{k}\OP^{2},\,G\bigr], & k \equiv 4 \pmod 8,\\[4pt]
      \bigl[\Sigma^{k}\OP^{2},\,G\bigr]\big/(\Z/2), & k \equiv 0 \pmod 8.
   \end{cases}
 \]
The isomorphism above may be chosen in such a way that each quadruple
$(s,t,t_{1},t_{2})\in\Z^{2}\oplus T_{k}\oplus\Theta_{k}$ corresponds to the
element of $\Sdiff_\partial(\OP^{2}\times\mathbb{D}^{k})$ with relative smooth
normal invariant
\[
s\,\widetilde{\eta_{1,k}}+t\,\widetilde{\eta_{2,k}}+\overline{\gamma_{*}}(t_{1})+\eta^{\diff}_{\mathbb{S}^{k}}(t_{2}),
\]
where $\{\widetilde{\eta_{1,k}},\widetilde{\eta_{2,k}}\}$ is a $\Z$-basis,
constructed in Proposition~\ref{prop:kernel-sigma}, of the complement of the
torsion subgroup in the kernel of the surgery obstruction map
$\sigma^{\diff}_{16+k}\colon
\Ndiff_{\partial}(\OP^{2}\times\mathbb{D}^{k})\rightarrow L_{16+k}(\Z)$;
$\overline{\gamma_{*}}\colon T_{k}\to[\Sigma^{k}\OP^{2},G/O]$ is the
homomorphism induced by $\gamma_{*}\colon[\Sigma^{k}\OP^{2},G]\to[\Sigma^{k}\OP^{2},G/O]$, where
$\gamma\colon G\to G/O$ is the canonical map; and
$\eta^{\diff}_{\mathbb{S}^{k}}\colon\Theta_{k}\to\pi_{k}(G/O)$ denotes the
Kervaire--Milnor homomorphism, the groups
$\bigl[\Sigma^{k}\OP^{2},G/O\bigr]$ and $\pi_{k}(G/O)$ being regarded as
subgroups of $\Ndiff_{\partial}(\OP^{2}\times\mathbb{D}^{k})$ via the
splitting~\eqref{eq:N-PT-splitting}.
\end{theorem}
The coefficients $s$ and $t$ are related to the underlying manifold through
the Pontryagin numbers computed in Theorem~\ref{thm:st-pontryagin}.

\begin{theorem}
\label{thm:st-pontryagin}
Let $k=4m$ with $m\ge1$, and let $[(M,\partial M),(f,\partial f)]
\in\Sdiff_\partial(\OP^{2}\times\mathbb{D}^{k})$ have relative smooth normal invariant $s\,\widetilde{\eta_{1,k}} +t\,\widetilde{\eta_{2,k}} +x$, where $s,t\in\mathbb{Z}$, $\{\widetilde{\eta_{1,k}},\widetilde{\eta_{2,k}}\}$ is a $\Z$-basis given in Proposition~\ref{prop:kernel-sigma}, and $x$ belongs to the torsion subgroup of
$\Ndiff_{\partial}(\OP^{2}\times\mathbb{D}^{k})$.

Let
\[
\widehat{M}
:=
M\cup_{\partial f}(\OP^{2}\times\mathbb{D}^{k})
\]
be the closed oriented $(16+k)$-manifold obtained by gluing
$\OP^{2}\times\mathbb{D}^{k}$ to $M$ along $\partial f$, with orientation chosen such that
$f\cup\operatorname{id}$ has degree one. For a closed oriented manifold $N$, write
\[
(\mathfrak{p}_{i}\mathfrak{p}_{j})(N)
:=
\langle
p_i(TN)\smile p_j(TN),
[N]
\rangle .
\]

Then the integers $s$ and $t$ are uniquely determined by the Pontryagin numbers of $\widehat{M}$.

If $k\neq8,16$, then
\begin{align}
(\mathfrak{p}_{4}\mathfrak{p}_{m})(\widehat{M})
&=
P_{m}\,t,
\label{eq:st-P1}\\
(\mathfrak{p}_{2}\mathfrak{p}_{m+2})(\widehat{M})
&=
Q_{m}\,s+T_{m}\,t.
\label{eq:st-P2}
\end{align}
If $k=16$, the same holds with the right-hand side of equation~\eqref{eq:st-P1}
doubled and equation~\eqref{eq:st-P2} unchanged; if $k=8$, both monomials
coincide with $\mathfrak{p}_{2}\,\mathfrak{p}_{4}$ and
equations~\eqref{eq:st-P1}--\eqref{eq:st-P2} are replaced by
\begin{equation*}
\mathfrak{p}_{2}^{3}\bigl(\widehat{M}\bigr)=R_{1}\,t,
\qquad
\bigl(\mathfrak{p}_{2}\,\mathfrak{p}_{4}\bigr)\bigl(\widehat{M}\bigr)
=Q_{m}\,s+R_{2}\,t.
\end{equation*}
Here $P_{m}$, $Q_{m}$, $T_{m}$, $R_{1}$, and $R_{2}$ are the integers defined in equation~\eqref{eq:st-constants}; moreover, $P_{m}$, $Q_{m}$, and $R_{1}$ are nonzero.
\end{theorem}
Theorem~\ref{thm:st-pontryagin}, combined with the topological invariance of
the rational Pontryagin classes \cite[Theorem~1]{Novikov}, detects
homeomorphism types among manifolds homotopy equivalent to
$\OP^{2}\times\mathbb{S}^{k}$.
\begin{corollary}\label{cor:infinitely-many}
Let $k=4m$ with $m\geq1$. There is a family $\{W_{n,k}\}_{n\in\Z}$ of closed
smooth $(16+k)$-manifolds, each homotopy equivalent to
$\OP^{2}\times\mathbb{S}^{k}$, such that
\[
\frac{(\mathfrak{p}_{2}\,\mathfrak{p}_{m+2})(W_{n,k})}
     {(\mathfrak{p}_{2}\,\mathfrak{p}_{m+2})(W_{n',k})}
=\frac{n}{n'}
\]
for all $n\in\Z$ and all nonzero $n'\in\Z$. In particular,
$\{W_{n,k}\}_{n\geq1}$ is an infinite family of pairwise non-homeomorphic
manifolds.
\end{corollary}

\begin{theorem}
\label{thm:image-obstruction}
Let $k = 4m$ with $m \geq 1$. The relative smooth surgery obstruction map
\[
\sigma^{\mathrm{DIFF}}_{16+k}\colon
\mathcal{N}^{\mathrm{DIFF}}_{\partial}(\OP^2\times \mathbb{D}^k)
\rightarrow
L_{16+k}(\mathbb{Z})\cong\mathbb{Z}
\]
vanishes on the torsion subgroup, 
and on a $\Z$-basis $\{\widetilde{\xi_{\mathbb{S}^{k}}},
\widetilde{\xi_{1,k}},
\widetilde{\xi_{2,k}}\}$ of $\mathcal{N}^{\mathrm{DIFF}}_{\partial}
(\OP^{2}\times\mathbb{D}^{k})_{\mathrm{free}}$ constructed in Corollary~\ref{cor:free-basis}, it is given by

\begin{align*}
\sigma^{\mathrm{DIFF}}_{16+k}(\widetilde{\xi_{\mathbb{S}^{k}}})
&= S_{m} := -\tfrac{\kappa_{m}}{8}\,D_{m}\,j_{m},\\
\sigma^{\mathrm{DIFF}}_{16+k}(\widetilde{\xi_{1,k}})
&= A_{m} := -\tfrac{\kappa_{m}}{8}\Bigl[\Bigl(\tfrac{14}{15}D_{m+2}
+\tfrac{D_{m+4}}{240}\Bigr)j_{m+2}+D_{m+4}\,c_{m}\Bigr],\\
\sigma^{\mathrm{DIFF}}_{16+k}(\widetilde{\xi_{2,k}})
&= U_{m} := -\tfrac{\kappa_{m}}{8}\,D_{m+4}\,j_{m+4},
\end{align*}
where $\kappa_{m}\in\{1,2\}$ is as in~\eqref{eq:kappa-def},
$D_{s}\in\Q$ as in Lemma~\ref{lem:L-eq-D-ph},
$j_{s}\in\Z_{>0}$ as in Theorem~\ref{prop:explicit-gens}, and
$c_{m}\in\Z$ as in congruence~\eqref{eq:cm-congruence}.
Consequently the image of $\sigma^{\mathrm{DIFF}}_{16+k}$ is the
subgroup $g_{m}\mathbb{Z}$, where $g_{m}=\gcd(S_{m},A_{m},U_{m})$, and, with $v_{2}$ denoting the
$2$-adic valuation,
\[
v_2(g_{m})=
\begin{cases}
\dfrac{k}{2}-2, & k\equiv 0\pmod 8,\\[6pt]
\dfrac{k}{2}-1, & k\equiv 4\pmod 8.
\end{cases}
\]
In particular $v_{2}(g_{m})\geq 1$, so that
$g_{m}\mathbb{Z}\subseteq 2\mathbb{Z}$ and $\sigma^{\mathrm{DIFF}}_{16+k}$
is not surjective. 
\end{theorem}

We record only the $2$-primary valuation of $g_{m}$; its behavior at odd primes is not determined in general (see the discussion following Lemma~\ref{lem:2-valuation}).
\begin{remark}
The failure of surjectivity in Theorem~\ref{thm:image-obstruction} may be
restated in terms of Wall's realization map. Writing
$\Delta\colon L_{16+k}(\mathbb{Z})\cong\mathbb{Z}\to
\mathcal{S}^{\mathrm{DIFF}}_{\partial}(\OP^2\times \mathbb{D}^{k-1})$
for the realization map in the relative smooth surgery exact sequence of
$\OP^2\times \mathbb{D}^{k-1}$, exactness gives
$\ker\Delta=\operatorname{im}\sigma^{\mathrm{DIFF}}_{16+k}=g_{m}\mathbb{Z}$. Therefore the image of $\Delta$ is a nonzero finite cyclic group of order $g_{m}$. In particular, since $g_{m}$ is even, $\Delta$ carries every odd class of
$L_{16+k}(\mathbb{Z})$ to a nontrivial element of
$\mathcal{S}^{\mathrm{DIFF}}_{\partial}(\OP^2\times \mathbb{D}^{k-1})$.
\end{remark}
A question motivated by the study of spaces of Riemannian metrics of
positive curvature is whether there exist smooth fibre bundles over
spheres with non-vanishing $\widehat{\mathfrak A}$-genus whose fibre
carries a metric of positive sectional curvature. If the fibre is
only required to carry a metric of positive \emph{scalar} curvature,
such bundles exist over base spheres of every dimension by
\cite[Corollary~1.6]{HSS14}; the construction used there, however,
yields no explicit description of the fibre \cite[p.~337]{HSS14}.
Strengthening positive scalar curvature to positive sectional or
Ricci curvature on the fibre was described as a ``very difficult
problem'' in \cite[p.~3]{botvik}. Since $\OP^{2}$ carries a metric
of positive sectional curvature, the following result produces
bundles of the desired kind.

\begin{theorem}\label{thm:Ahat-OP2-bundle}
For each $k \in \{4,\, 8,\, 12\}$ there exists a smooth oriented fibre bundle
\[
   \OP^2 \;\rightarrow\; E \;\rightarrow\; \mathbb{S}^k
\]
whose total space has non-vanishing
$\widehat{\mathfrak{A}}$-genus.
\end{theorem}

This answers the question above for the base spheres
$\mathbb{S}^{4}$, $\mathbb{S}^{8}$, and $\mathbb{S}^{12}$. In the case
$k=4$, the bundle is the octonionic analogue of the $\HP^{2}$-bundle
studied in \cite{KKR21}, while the cases $k=8,12$ are new.

As a further consequence, the spaces of positive sectional, Ricci, and scalar curvature metrics on $\OP^2$ each carry an element of infinite order in their homotopy groups in degrees $3$, $7$, and $11$ (see Corollary~\ref{cor:positive-curvature}).

The proof of Theorem~\ref{thm:Ahat-OP2-bundle} is given in
Section~\ref{sec:applications} and proceeds as follows. We first produce a stable vector bundle $\xi$ over the suspension
$\Sigma^{k}\OP^{2}$ admitting a lift $[\widetilde{\xi}]$ to the relative normal
invariant set $\Ndiff_{\partial}(\OP^{2} \times \mathbb{D}^{k})$ with vanishing
surgery obstruction, while the closed manifold $\widehat{M}$, obtained by
attaching $\OP^{2} \times \mathbb{D}^{k}$ along the boundary of the domain of a
normal map representing $[\widetilde{\xi}]$, has non-vanishing
$\widehat{\mathfrak{A}}$-genus. Realizing $\widehat{M}$ as the total space of a
block $\OP^{2}$-bundle over $\mathbb{S}^{k}$ produces an element of infinite
order in the $k$-th homotopy group of the classifying space of the block
diffeomorphism group of $\OP^{2}$. The results of
Section~\ref{sec:applications} then allow us to replace this block bundle by a
smooth $\OP^{2}$-bundle over $\mathbb{S}^{k}$.

\subsection*{Organization}
Section~\ref{sec:normal} computes
$\Ndiff_{\partial}(\OP^{2}\times\mathbb{D}^{k})$ from the
$KO$-theory and Adams operations of $\Sigma^{k}\OP^{2}$, and
exhibits an explicit basis of its free part.
Section~\ref{sec:obstr-free} evaluates the surgery obstruction on
this basis by means of the Hirzebruch signature theorem and the
Pontryagin character, and describes its kernel.
Section~\ref{sec:proofs} assembles these computations into the
proofs of Theorems~\ref{thm:main-structure},
\ref{thm:st-pontryagin} and~\ref{thm:image-obstruction}, and of
Corollary~\ref{cor:infinitely-many}.
Section~\ref{sec:applications} contains the applications to
diffeomorphism groups, to $\OP^{2}$-bundles with non-vanishing
$\widehat{\mathfrak{A}}$-genus, and to spaces of positively curved
metrics.
Appendix~\ref{sec:appendix1} collects the arithmetic properties of
Bernoulli numbers used in Sections~\ref{sec:normal}
and~\ref{sec:proofs}.




\section{Surgery-theoretic preliminaries and the normal invariants}\label{sec:normal}
This section sets up the surgery-theoretic framework for
$\OP^{2}\times\mathbb{D}^{k}$ and computes its relative normal
invariants. After recalling the relative structure set and the
surgery exact sequence, we determine the group
$\Ndiff_{\partial}(\OP^{2}\times\mathbb{D}^{k})$ (Theorem~\ref{thm:normal-inv-decomp}) and exhibit an explicit
$\Z$-basis of its free part (Corollary~\ref{cor:free-basis}); the evaluation of the surgery obstruction in
Section~\ref{sec:obstr-free} is carried out with respect to this
basis.
\begin{definition}
Let $X$ be a compact smooth manifold with boundary $\partial X$. The
\emph{relative smooth structure set} $\Sdiff_{\partial}(X)$ is the
pointed set of equivalence classes of homotopy equivalences of pairs
$(f,\partial f)\colon (M,\partial M)\rightarrow (X,\partial X)$,
where $M$ is a compact smooth manifold with boundary $\partial\,M$, and $\partial f$ is a diffeomorphism. Two such
representatives $(f_{1},\partial f_{1})$ and $(f_{2},\partial f_{2})$ are
equivalent if there exists a diffeomorphism
$(g,\partial g)\colon (M_{1},\partial M_{1})\to(M_{2},\partial M_{2})$
making the diagram
\[
\begin{tikzcd}[column sep=large]
(M_{1},\partial M_{1})
\arrow[r,"{(f_{1},\partial f_{1})}"]
\arrow[d,"{(g,\partial g)}"',"\cong"]
& (X,\partial X) \\
(M_{2},\partial M_{2})
\arrow[ur,"{(f_{2},\partial f_{2})}"']
&
\end{tikzcd}
\]
commute up to homotopy relative to $\partial M_{1}$. The basepoint
is represented by $\operatorname{id}\colon X\to X$. If
$\partial X=\emptyset$, we write $\Sdiff(X)$. For $X=Y\times\mathbb{D}^{k}$
with $k\geq 1$, the set $\Sdiff_{\partial}(Y\times\mathbb{D}^{k})$ is
a group, abelian for $k\geq 2$, under
stacking~\cite[\S 11.8]{LM24}.
\end{definition}

Throughout the remainder of the article, we work in the smooth
category; thus all manifolds, isomorphisms, and structure sets are
understood to be smooth unless stated otherwise. 
Since $\dim(\OP^{2}\times\mathbb{D}^{k})=16+k\ge5$, the relative surgery exact sequence \cite[Chapter~11]{LM24} yields the following exact sequence of abelian groups:

\begin{equation}\label{eq:surgery-even}
   L_{17+k}(\Z)=0 \to \Sdiff_\partial(\OP^2 \times \mathbb{D}^k)
     \xrightarrow{\;\eta_{\partial}^{\diff}\;}
     \Ndiff_\partial(\OP^2 \times \mathbb{D}^k)
     \xrightarrow{\;\sigma_{16+k}^{\diff}\;}
     L_{16 + k}(\Z) = \Z\to \cdots.
\end{equation}
Here $\Ndiff_\partial(\OP^2 \times \mathbb{D}^k)$ consists of normal
bordism classes of degree-one normal maps
$(M,\partial M)\to(\OP^{2}\times\mathbb{D}^{k},
\OP^{2}\times\mathbb{S}^{k-1})$ restricting to diffeomorphisms on the
boundary, the normal invariant map $\eta_{\partial}^{\diff}$ assigns to a
homotopy equivalence $(f,\partial f)$ its normal invariant, and, as
$16+k \equiv 0 \pmod{4}$, the surgery obstruction
$\sigma_{16+k}^{\diff}$ is one eighth of the difference of the
signatures of the two closed manifolds obtained from the source and
the target of a representing normal map by gluing
$\OP^{2}\times\mathbb{D}^{k}$ along the boundary;
see~\eqref{eq:browder-additivity}. By exactness of~\eqref{eq:surgery-even}, $\eta_{\partial}^{\diff}$ maps
$\Sdiff_\partial(\OP^2 \times \mathbb{D}^k)$ isomorphically onto
$\mathrm{Ker}(\sigma_{16+k}^{\diff})$.
The computation thus divides into two steps: the determination of the
normal invariants, carried out in the present section, and the
analysis of the surgery obstruction, carried out in
Section~\ref{sec:obstr-free}.

Recall that $G/O$ denotes the homotopy fibre of the map of classifyin
spaces $J\colon BO\to BG$ representing the spherical fibration
associated to the universal bundle over $BO$; we write
$r\colon G/O\to BO$ for the fibre inclusion. For any compact smooth manifold $X$, Sullivan's identification yields
$\Ndiff_{\partial}(X)\cong[X/\partial X,\,G/O]$ \cite[Chapter~10]{Wal99}. For the rest of the article, let
\[
X_k:=\OP^{2}\times\mathbb{D}^{k}\big/\OP^{2}\times\mathbb{S}^{k-1}\simeq \Sigma^k \OP^2\vee \mathbb{S}^k.
\]
Hence,
\begin{equation}\label{eq:N-PT-splitting}
\Ndiff_{\partial}(\OP^{2}\times\mathbb{D}^{k}) \cong [X_k, G/O]
\;\cong\;
\bigl[\Sigma^{k}\OP^{2},\,G/O\bigr]\oplus\pi_{k}(G/O).
\end{equation}

Throughout the article, we adopt the following notation and conventions.

Fix a generator $u\in H^{8}(\OP^{2};\Z)$ such that
$H^{*}(\OP^{2};\Z)=\Z[u]/(u^{3})$. The CW decomposition
$\OP^{2}=\mathbb{S}^{8}\cup_{\sigma}e^{16}$ induces the cofibre sequence
\begin{equation}\label{eq:cofibre}
\mathbb{S}^{15}\xrightarrow{\sigma}\mathbb{S}^{8}\xhookrightarrow{i}
\OP^{2}
\xrightarrow{q}\mathbb{S}^{16},
\end{equation}
where $i$ is the inclusion of the $8$-skeleton and $q$ the quotient map. Let
$\iota_k\in H^{k}(\mathbb{S}^{k};\Z)$ denote the orientation class. Identifying
$\Sigma^{k}\OP^{2}\simeq \OP^{2}\wedge\mathbb{S}^{k}$, the suspension isomorphism is
given by
\[
\sigma^{k}x=x\wedge\iota_k,
\]
and we write integral and rational cohomology classes using the same notation. Thus,
\[
\widetilde H^{*}(\Sigma^{k}\OP^{2};R)
=R\{\sigma^{k}u,\sigma^{k}u^{2}\},
\qquad R\in\{\Z,\Q\}.
\]


Reduced $KO$-groups are regarded as $KO^{*}(\mathrm{pt})$-modules via the
external product. Under the identification $\widetilde{KO}^{-k}(\mathbb{S}^{n})
\cong
KO^{-(k+n)}(\mathrm{pt})$, external products correspond to multiplication in the coefficient ring.
By \cite[p.~63]{spingeo},
\[
KO^{*}(\mathrm{pt})
=
\Z[\eta,\mu,b^{\pm1}]
/(2\eta,\eta^{3},\mu\eta,\mu^{2}-4b),
\]
where $|\eta|=-1$, $|\mu|=-4$, and $|b|=-8$. Hence, for $n\ge0$,
\begin{equation}\label{eq:KO-pt-groups}
\pi_{n}(BO)=KO^{-n}(\mathrm{pt})=
\begin{cases}
\Z\{b^{n/8}\} & n \equiv 0 \pmod{8},\\
\Z/2\{\eta\, b^{(n-1)/8}\} & n \equiv 1 \pmod{8},\\
\Z/2\{\eta^{2} b^{(n-2)/8}\} & n \equiv 2 \pmod{8},\\
\Z\{\mu\, b^{(n-4)/8}\} & n \equiv 4 \pmod{8},\\
0 & \text{otherwise}.
\end{cases}
\end{equation}
For $n\equiv0,1,2,4\pmod8$, let $\beta_n\in KO^{-n}(\mathrm{pt})$
denote the generator in~\eqref{eq:KO-pt-groups}. By suspension we use
the same notation for the corresponding generator of
$\widetilde{KO}^{0}(\mathbb{S}^{n})$; in particular,
$\beta_8=b$ and $\beta_{16}=b^2$. Since each $\beta_n$ is a monomial in
$\eta$, $\mu$, and $b$,
\begin{equation}\label{eq:beta-products}
\beta_{n}\,\beta_{8}=\beta_{n+8},
\qquad
\beta_{n}\,\beta_{16}=\beta_{n+16}
\end{equation}
whenever $KO^{-n}(\mathrm{pt})\neq 0$.

The Pontryagin character is the composite $\ph=\operatorname{ch}\circ c:KO^{0}(X)\to H^{4*}(X;\Q)$, where $c$ denotes complexification. It induces an isomorphism after tensoring with $\Q$ for every finite CW complex $X$ \cite[p.~7]{kotheory}. We write $\widetilde{\ph}$ and
$\widetilde{\operatorname{ch}}$ for the reduced Pontryagin and Chern
characters.

For $t\ge1$, let
\begin{equation}\label{eq:kappa-def}
\kappa_{t}:=
\begin{cases}
1 & \text{if } t \text{ is even},\\
2 & \text{if } t \text{ is odd},
\end{cases}
\end{equation}
where $\kappa_t$ is the degree of $c:KO^{-4t}(\mathrm{pt})\to KU^{-4t}(\mathrm{pt})$ \cite[p.~618]{Ada62}.
For $n\equiv 0,4\pmod 8$, the group in~\eqref{eq:KO-pt-groups} is
infinite cyclic, and the generator $\beta_{n}$ is therefore determined
only up to sign.
We may therefore assume that
\[
\ph(\beta_{8})=i^{*}u,
\qquad
q^{*}\ph(\beta_{16})=u^{2};
\]
indeed, complexification is an isomorphism on
$\widetilde{KO}^{0}(\mathbb{S}^{8t})$ and the Chern character maps a
generator of $\widetilde{KU}^{0}(\mathbb{S}^{8t})$ to a generator of
$H^{8t}(\mathbb{S}^{8t};\Z)$.


Let $B_{s}$ denote the $s$-th Bernoulli number in the classical
convention of \cite[p.~90]{hardy}, which is also that of
\cite[Appendix~B]{milnor}; in this convention every $B_{s}$ is positive,
and the first few values are
\[
B_{1}=\tfrac{1}{6},\qquad B_{2}=\tfrac{1}{30},\qquad
B_{3}=\tfrac{1}{42},\qquad B_{4}=\tfrac{1}{30},\quad\dots.
\]
\subsection{\texorpdfstring{$KO$}{}-theory of \texorpdfstring{$\Sigma^{k}\OP^2$}{} and Adams operations}
In this subsection, we compute $\widetilde{KO}^{-k}(\OP^{2})$ for
$k\equiv 0\pmod 4$ and for $k\equiv 1\pmod 8$
(Lemma~\ref{lem:module-basis}), and determine the corresponding
Adams operations (Propositions~\ref{prop:adams-ops}
and~\ref{prop:adams-ops-1mod8}). The first case is what the
computation of the normal invariants rests on; the second is needed only for Proposition~\ref{kerk1}, which is
used in Lemma~\ref{cok} when $k\equiv 0\pmod 8$.
\begin{lemma}\label{lem:KO0-ring}
There exists $\alpha\in\widetilde{KO}^{0}(\OP^{2})$ with
$i^{*}(\alpha)=\beta_{8}$. For every such $\alpha$,
\[\alpha^{2}=q^{*}(\beta_{16}),
\qquad
\alpha^{3}=0,
\qquad
\widetilde{KO}^{0}(\OP^{2})
=\Z\{\alpha,\alpha^{2}\},
\]
and hence $KO^{0}(\OP^{2})
\cong
\Z[\alpha]/(\alpha^{3})$ as rings.
\end{lemma}
\begin{proof}
Applying $[-, BO]$ to the cofibre sequence~\eqref{eq:cofibre} yields
\begin{equation}\label{BO}
   \pi_{9+k}(BO)
   \xrightarrow{\;(\Sigma^{k+1} \sigma)^*\;}
   \pi_{16+k}(BO)
   \xrightarrow{\;(\Sigma^k q)^*\;}
   [\Sigma^k\OP^2, BO]
   \xrightarrow{\;(\Sigma^k i)^*\;}
   \pi_{8+k}(BO)
   \xrightarrow{\;(\Sigma^k \sigma)^*\;}
   \pi_{15+k}(BO).
\end{equation}
By~\eqref{eq:KO-pt-groups}, the exact sequence~\eqref{BO} takes the form
\begin{equation}\label{eq:KO0-SES}
0 \to \widetilde{KO}^{0}(\mathbb{S}^{16})
\xrightarrow{\;q^{*}\;}
\widetilde{KO}^{0}(\OP^{2})
\xrightarrow{\;i^{*}\;}
\widetilde{KO}^{0}(\mathbb{S}^{8})
\to 0,
\end{equation}
which splits since $\widetilde{KO}^{0}(\mathbb{S}^{8})\cong\mathbb Z$ is free. Hence there exists $\alpha\in\widetilde{KO}^{0}(\OP^{2})$ with $i^{*}(\alpha)=\beta_{8}$.

Since $\widetilde{KO}^{0}(\OP^{2})$ is torsion-free by
\eqref{eq:KO0-SES} and $\ph\otimes\Q$ is an isomorphism, $\ph$ is
injective on $\widetilde{KO}^{0}(\OP^{2})$. By the naturality of
$\ph$, the identity $i^{*}(\alpha)=\beta_{8}$, and the assumption
$\ph(\beta_{8})=i^{*}u$, we obtain
\[
i^{*}\ph(\alpha)=\ph(i^{*}\alpha)=\ph(\beta_{8})=i^{*}u.
\]
Since $\ker(i^{*}\otimes \Q)=\Q\{u^{2}\}$, we have
\begin{equation}\label{eq:ph-alpha}
\ph(\alpha)=u+\lambda u^{2}
\end{equation}
for some $\lambda\in\Q$. Using the ring structure of $H^*(\OP^2;\Q)$, we get
\[
\ph(\alpha^{2})
=(u+\lambda u^{2})^{2}
=u^{2}
=\ph(q^{*}\beta_{16}),
\qquad
\ph(\alpha^{3})
=(u+\lambda u^{2})^{3}
=0.
\]
Now the injectivity of $\ph$ gives $\alpha^{2}=q^{*}(\beta_{16})$ and $\alpha^{3}=0$. Finally,
$q^{*}$ identifies
$\widetilde{KO}^{0}(\mathbb{S}^{16})$
with $\Z\{\alpha^{2}\}$, while $\beta_{8}\mapsto\alpha$ defines a
section of $i^{*}$ in~\eqref{eq:KO0-SES}. Thus $\widetilde{KO}^{0}(\OP^{2}) \cong \Z\{\alpha\}\oplus\Z\{\alpha^{2}\}$,
completing the proof.

\end{proof}
\begin{remark}
An explicit class $\alpha$ satisfying the hypothesis of Lemma~\ref{lem:KO0-ring} is
$16-[T\OP^{2}]$, the reduced class of the stable normal bundle of $\OP^{2}$.
Indeed, $p_{1}(\OP^{2})$ vanishes for degree reasons, while
$p_{2}(\OP^{2})=6u$ by \cite[Theorem~19.4]{borel}. For a real bundle $\eta$ with $p(\eta)=\prod_{i}(1+x_{i}^{2})$, the complexification
$\eta\otimes_{\mathbb{R}}\mathbb{C}$ has Chern roots $\pm x_{i}$. Hence $\ph(\eta)=\sum_{i}\bigl(e^{x_{i}}+e^{-x_{i}}\bigr)=\sum_{i}2\cosh x_{i}$, whose degree-$8$ component is $\tfrac{1}{12}\sum_{i}x_{i}^{4}=\tfrac{1}{12}\bigl(p_{1}^{2}-2p_{2}\bigr)$.
The degree-$8$ component of $\ph(T\OP^{2})$ is therefore $-u$, and hence $\widetilde{\ph}\bigl(16-[T\OP^{2}]\bigr)=u+c\,u^{2}$ for some $c\in\Q$. Using the naturality of $\ph$ together with $i^{*}(u^{2})=0$ and $\ph(\beta_{8})=i^{*}u$, we obtain
\[
\ph\Bigl(i^{*}\bigl(16-[T\OP^{2}]\bigr)\Bigr)
=i^{*}\,\widetilde{\ph}\bigl(16-[T\OP^{2}]\bigr)
=i^{*}u=\ph(\beta_{8}).
\]
Since $\widetilde{KO}^{0}(\mathbb{S}^{8})\cong\Z$ is torsion-free and
$\ph\otimes\Q$ is an isomorphism \cite[p.~7]{kotheory}, the map $\ph$ is
injective on $\widetilde{KO}^{0}(\mathbb{S}^{8})$, whence
$i^{*}\bigl(16-[T\OP^{2}]\bigr)=\beta_{8}$.
\end{remark}
For $k \geq 1$ with $KO^{-k}(\mathrm{pt}) \neq 0$ and $\alpha$ as in
Lemma~\ref{lem:KO0-ring}, define
\begin{equation*}
\alpha_{k} := \beta_{k} \cdot \alpha,
\qquad
\alpha_{k}^{2} := \beta_{k} \cdot \alpha^{2}
\;\in\; \widetilde{KO}^{-k}(\OP^{2}) \cong [\Sigma^{k}\OP^{2}, BO],
\end{equation*}
the products being taken with respect to the
$KO^{*}(\mathrm{pt})$-module structure. Since $\alpha^{2}=q^{*}(\beta_{16})$ does not depend on the choice
of $\alpha$, neither does $\alpha_{k}^{2}$; the class $\alpha_{k}$, however,
does.

\begin{lemma}\label{lem:module-basis}\label{lem:KO-OP2}\label{1mod8}
Let $k \geq 1$.
\begin{enumerate}
\item[(a)] If $k \equiv 0 \pmod{4}$, then
$[\Sigma^{k}\OP^{2}, BO] = \Z\{\alpha_{k}\} \oplus \Z\{\alpha_{k}^{2}\}$.
\item[(b)] If $k \equiv 1 \pmod{8}$, then
$[\Sigma^{k}\OP^{2}, BO] = \Z/2\{\alpha_{k}\} \oplus \Z/2\{\alpha_{k}^{2}\}$.
\end{enumerate}
\end{lemma}
\begin{proof}
$KO^{*}(\mathrm{pt})$-linearity of $i^{*}$ and $q^{*}$,
Lemma~\ref{lem:KO0-ring}, and~\eqref{eq:beta-products} give
\[
i^{*}(\alpha_{k}) = \beta_{k} \cdot i^{*}(\alpha) = \beta_{k}\beta_{8}
= \beta_{k+8},
\qquad
\alpha_{k}^{2} = \beta_{k} \cdot q^{*}(\beta_{16})
= q^{*}(\beta_{k}\beta_{16}) = q^{*}(\beta_{k+16}).
\]
Suppose $k \equiv 0 \pmod{4}$. Then by~\eqref{eq:KO-pt-groups}, we have $\pi_{8+k}(BO)\cong\pi_{16+k}(BO)\cong\Z$ and
$\pi_{15+k}(BO)=0$. Moreover, note that for $k\equiv 0\pmod{4}$, $\pi_{9+k}(BO)$ is either trivial or contains $2$-torsion. This reduces the exact sequence~\eqref{BO}  to the following short exact sequence
\[
   0 \to \pi_{16+k}(BO)\cong \Z
   \xrightarrow{\;(\Sigma^k q)^*\;}
   [\Sigma^k\OP^2, BO]
   \xrightarrow{\;(\Sigma^k i)^*\;}
   \pi_{8+k}(BO)\cong\Z \to 0.
\]

Suppose $k\equiv1\pmod8$. Since $\sigma^{*}\colon\pi_{9+k}(BO)\to\pi_{16+k}(BO)$ and $\sigma^{*}\colon\pi_{8+k}(BO)\to\pi_{15+k}(BO)$ are trivial by \cite[Proposition~7.1]{Ada66}, and
$\pi_{8+k}(BO)\cong\pi_{16+k}(BO)\cong\Z/2$, the exact sequence
\eqref{BO} reduces to
\begin{equation*}
0\rightarrow
\pi_{16+k}(BO)\cong \Z/2
\xrightarrow{\,(\Sigma^kq)^{*}\,}
[\Sigma^{k}\OP^{2},BO]
\xrightarrow{\,(\Sigma^k i)^{*}\,}
\pi_{8+k}(BO)\cong \Z/2
\rightarrow0.
\end{equation*}
In either case, $(\Sigma^{k}q)^{*}$ is injective with image
$\Z\{\alpha_{k}^{2}\}$, where
$\alpha_{k}^{2}=(\Sigma^{k}q)^{*}(\beta_{k+16})$, and
$(\Sigma^{k}i)^{*}(\alpha_{k})=\beta_{k+8}$. Hence
$\beta_{k+8}\mapsto\alpha_{k}$ defines a section
$s$ of $(\Sigma^{k}i)^{*}$: for $k\equiv0\pmod4$ this is immediate,
while for $k\equiv1\pmod8$ it follows from
$2\alpha_{k}=(2\beta_{k})\cdot\alpha=0$. Therefore
$[\Sigma^{k}\OP^{2},BO]
=
\Z\{\alpha_{k}^{2}\}\oplus\Z\{\alpha_{k}\}$, where each summand is infinite cyclic, if $k\equiv0\pmod4$ and cyclic of order $2$, if $k\equiv1\pmod8$.

\end{proof}
\begin{remark}
The notation $\alpha_k^2$ is purely symbolic: it denotes the external
product $\beta_k\cdot\alpha^2$, not the square of $\alpha_k$. Indeed,
$(\alpha_k)^2=0$, since the reduced diagonal of a suspension is
null-homotopic.
\end{remark}
Consider the complexification
$c \colon \widetilde{KO}^{0}(\OP^{2}) \to \widetilde{KU}^{0}(\OP^{2})$
and set $\beta := c(\alpha)$. By~\eqref{eq:ph-alpha},
\[
\ch(\beta) = u + \lambda u^{2},
\qquad
\ch(\beta^{2}) = (u + \lambda u^{2})^{2} = u^{2},
\]
where $\lambda \in \Q$ depends on the choice of $\alpha$. The class of $\lambda$ in $\mathbb{Q}/\mathbb{Z}$ is equal to that of $\frac{1}{m(4)}$ \cite[Proposition~7.9]{Ada66}, where $m(4)$, which equals $240$, denotes the denominator of the fraction $\frac{B_{2}}{8}$ (see \cite[Theorem~2.6]{Ada65a}). Write $\lambda = \frac{1}{240} + n$
with $n \in \Z$. The class $\alpha - n\alpha^{2}$ satisfies
$i^{*}(\alpha - n\alpha^{2}) = \beta_{8}$, since
$i^{*}(\alpha^{2}) = 0$, and
$(\alpha - n\alpha^{2})^{2} = \alpha^{2}$, since $\alpha^{3} = 0$;
hence Lemma~\ref{lem:KO0-ring} and Lemma~\ref{lem:module-basis}
apply verbatim to $\alpha - n\alpha^{2}$, while
$\ch\bigl(c(\alpha - n\alpha^{2})\bigr)
= u + \lambda u^{2} - n u^{2} = u + \tfrac{1}{240}\,u^{2}$.
Replacing $\alpha$ by $\alpha - n\alpha^{2}$, we may therefore
assume
\begin{equation}\label{eqn: ch at beta}
\ch(\beta)=u+\frac{1}{240}\,u^{2}.
\end{equation}
This normalization of $\alpha$ is fixed for the remainder of the
paper; in particular, the classes $\alpha_{k}$ and $\alpha_{k}^{2}$
of Lemma~\ref{lem:module-basis} are henceforth understood with
respect to this choice.

Let $k=4m$ with $m\ge1$. The following proposition determines the Adams operations on the generators $\alpha_k$ and $\alpha_k^2$ of
$[\Sigma^k\OP^2,BO]$.
\begin{proposition}\label{prop:adams-ops}
Let $k = 4m$ for some integer $m \geq 1$. Then, for every integer $j \geq 1$, the Adams operation
\[
\Psi^j \colon \widetilde{KO}^{-k}(\OP^2) \rightarrow \widetilde{KO}^{-k}(\OP^2)
\]
satisfies
\begin{equation}\label{eq:adams-suspended}
\begin{aligned}
   \Psi^j(\alpha_k)
   &= j^{4 + k/2}\, \alpha_k
     + \frac{j^{4 + k/2}(j^4 - 1)}{240}\, \alpha^2_k, \\
   \Psi^j(\alpha^2_k)
   &= j^{8 + k/2}\, \alpha^2_k.
\end{aligned}
\end{equation}
\end{proposition}
\begin{remark}\label{rem:240-divisibility}
For every integer $j$, $240\mid j^{4}(j^{4}-1)$. Indeed, $3$ and $5$ divide $j^{4}(j^{4}-1)$ by Fermat's little theorem,
while $16\mid j^{4}$ if $j$ is even and $16\mid(j^{4}-1)$ if $j$ is odd,
since $j^{2}\equiv1\pmod8$. Hence the coefficient
$\frac{j^{4+k/2}(j^{4}-1)}{240}$ in~\eqref{eq:adams-suspended} is an
integer. If $j$ is odd, then $j^{4}$ is odd, so $\frac{j^{4}(j^{4}-1)}{240}$ is even if and only if $32\mid(j^{4}-1)$. Now
$j^{2}\equiv1\pmod{16}$ for $j\equiv\pm1\pmod8$, whereas
$j^{2}\equiv9\pmod{16}$ for $j\equiv\pm3\pmod8$. Hence
$j^{4}\equiv1\pmod{32}$ in the former case and
$j^{4}\equiv17\pmod{32}$ in the latter, so
$j^{4}(j^{4}-1)/240$ is even for $j\equiv\pm1\pmod8$ and odd for
$j\equiv\pm3\pmod8$.
\end{remark}
\begin{proof}[Proof of Proposition~\ref{prop:adams-ops}]
We note from \cite[Proposition~7.5]{Ada66} that
\[
\Psi^{j}(\alpha)=j^{4}\alpha+ \frac{j^8-j^4}{240}\,\alpha^{2}
\qquad\text{and}\qquad
\Psi^j(\alpha^{2}) = j^8\alpha^{2}.
\]
Since $\Psi^j$ is a ring endomorphism of $KO^*(-)$
\cite[Theorem~5.1]{Ada62} and
$\Psi^j(\beta_k) = j^{k/2}\beta_k$
\cite[Corollary~5.2]{Ada62},
\begin{align*}
   \Psi^j(\alpha_k)
   &= \Psi^j(\beta_k) \cdot \Psi^j(\alpha)
    = j^{k/2}\beta_k \cdot \bigl(j^4\alpha + \frac{j^8-j^4}{240}\,\alpha^{2}\bigr)
  = j^{4+k/2}\alpha_k + j^{k/2}\frac{j^8-j^4}{240}\,\alpha^2_k.
\end{align*}
Therefore $\Psi^j(\alpha_k)=j^{4+k/2}\alpha_k+ \frac{j^{4+k/2}(j^4-1)}{240}\alpha_k^2$.

Similarly, we have $\Psi^j(\alpha_k^2)= j^{k/2}\beta_k \cdot j^8\alpha^{2}
= j^{8+k/2}\alpha^2_k$.
\end{proof}
We now compute the Adams operations on the
generators $\alpha_{k}$ and $\alpha_{k}^{2}$ of $[\Sigma^k \OP^2, BO]$, if $k = 8\ell + 1$ with $\ell \geq 0$. 

\begin{proposition}\label{prop:adams-ops-1mod8}
For \(k=8\ell+1\) with \(\ell\geq 0\), the Adams operations on $\widetilde{KO}^{-k}(\mathbb{O}P^2)$ are given as follows.
\begin{itemize}
\item[(a)] If \(j\) is odd, then $\Psi^j(\alpha_k)=\alpha_k+\delta(j)\alpha_k^2,~
\Psi^j(\alpha_k^2)=\alpha_k^2$,
where
\[
\delta(j):=\frac{j^4(j^4-1)}{240}\pmod{2}
=
\begin{cases}
0 & \text{if } j\equiv \pm1 \pmod 8,\\
1 & \text{if } j\equiv \pm3 \pmod 8.
\end{cases}
\]

\item[(b)] If \(j\) is even, then $\Psi^j(\alpha_k)=\Psi^j(\alpha_k^2)=0$.
\end{itemize}
\end{proposition}
\begin{proof}
Since $\Psi^{j}(b)=j^{4}b$ \cite[Corollary~5.2]{Ada62}, we have
$\Psi^{j}(b^{\ell})=j^{4\ell}b^{\ell}$. Let $L\to \mathbb{S}^{1}$ be the
M\"obius line bundle and set
$\eta:=L-\underline{\mathbb{R}}\in\widetilde{KO}^{0}(\mathbb{S}^{1})
\cong\Z/2$, representing the generator $\eta$ of
\eqref{eq:KO-pt-groups} \cite{hatcher2003}.
By \cite[Theorem~5.1(iii)]{Ada62} we have $\Psi^{j}(L)=L^{\otimes j}$,
and since $L^{\otimes 2}\cong\underline{\mathbb{R}}$, this gives
$L^{\otimes j}=L$ for $j$ odd and
$L^{\otimes j}=\underline{\mathbb{R}}$ for $j$ even. Hence
\[
\Psi^{j}(\eta)=
\begin{cases}
\eta, & j \text{ odd},\\
0, & j \text{ even}.
\end{cases}
\]
Since $\beta_{k}=\eta\, b^{\ell}$ and the Adams operations are
multiplicative with respect to the external product,
\[
\Psi^{j}(\beta_{k})=\Psi^{j}(\eta)\cdot j^{4\ell}b^{\ell}
=
\begin{cases}
\beta_{k}, & j \text{ odd},\\
0, & j \text{ even},
\end{cases}
\]
where in the odd case we used that $j^{4\ell}$ is odd and $\beta_{k}$
has order $2$.

If $j$ is even, then $\Psi^{j}(\beta_{k})=0$, and hence
$\Psi^{j}(\alpha_{k})=\Psi^{j}(\alpha_{k}^{2})=0$ by multiplicativity;
this proves (b).

Suppose $j$ is odd. Using \cite[Proposition~7.5]{Ada66} and
multiplicativity again, we obtain
\begin{align*}
\Psi^{j}(\alpha_{k})
&=\Psi^{j}(\beta_{k})\cdot\Psi^{j}(\alpha)
=\beta_{k}\Bigl(j^{4}\alpha+\tfrac{j^{4}(j^{4}-1)}{240}\,\alpha^{2}\Bigr)
=j^{4}\alpha_{k}+\tfrac{j^{4}(j^{4}-1)}{240}\,\alpha_{k}^{2},\\
\Psi^{j}(\alpha_{k}^{2})
&=\Psi^{j}(\beta_{k})\cdot\Psi^{j}(\alpha^{2})
=j^{8}\alpha_{k}^{2}.
\end{align*}
Since $\alpha_{k}$ and $\alpha_{k}^{2}$ have order $2$ and $j^{4}$,
$j^{8}$ are odd, these formulas reduce to
\[
\Psi^{j}(\alpha_{k})=\alpha_{k}+\delta(j)\,\alpha_{k}^{2},
\qquad
\Psi^{j}(\alpha_{k}^{2})=\alpha_{k}^{2},
\]
where $\delta(j)\in\{0,1\}$ is the residue of $j^{4}(j^{4}-1)/240$
modulo $2$; by Remark~\ref{rem:240-divisibility} this quotient is an
integer, even for $j \equiv \pm 1 \pmod{8}$ and odd for
$j \equiv \pm 3 \pmod{8}$. This proves (a).
\end{proof}
\subsection{Computation of\texorpdfstring{ $\Ndiff_{\partial}(\OP^2\times\mathbb{D}^{k})$}{}}
The aim of this subsection is to determine the algebraic structure of
$\Ndiff_{\partial}(\OP^{2}\times\mathbb{D}^{k})$.
The fibration $G/O\to BO\xrightarrow{\,J\,}BG$ yields a split
short exact sequence relating $[\Sigma^{k}\OP^{2},G/O]$ to
$\operatorname{Ker}[\Sigma^{k}\OP^{2},J]$ and
$\operatorname{Coker}[\Sigma^{k}\OP^{2},\Omega J]$
(Lemma~\ref{lem:normal-splitting}); the cokernel is computed in
Lemma~\ref{cok}, and the resulting decomposition is
Theorem~\ref{thm:normal-inv-decomp}. The kernel is shown here only to be free of rank two, explicit
generators for it being constructed in
Subsection~\ref{subsec:free-basis}.
\begin{lemma}\label{lem:normal-splitting}
For $k \equiv 0 \pmod{4}$ and $k \geq 1$,
\[
   \bigl[\Sigma^k \OP^2,\, G/O\bigr]
   \;\cong\;
   \Z^2 \;\oplus\;
   \operatorname{Coker}\bigl[\Sigma^k \OP^2, \Omega J\bigr],
\] where \(J\colon BO\to BG\) denotes the canonical map.
\end{lemma}
\begin{proof}
The fibre sequence
$O \xrightarrow{\Omega J} G \xrightarrow{\gamma} G/O \xrightarrow{r} BO
\xrightarrow{J} BG$
     induces the following short exact sequence
\begin{equation}\label{eq:N-SES}
   0 \to \operatorname{Coker}\bigl[\Sigma^k \OP^2, \Omega J\bigr]
        \xrightarrow{\;\overline{\gamma_{*}}\;}
       \bigl[\Sigma^k \OP^2, G/O\bigr]
        \xrightarrow{\;r_*\;}
       \operatorname{Ker}\bigl[\Sigma^k \OP^2, J\bigr]
   \to 0,
\end{equation}
where $\overline{\gamma_{*}}$ is induced by the homomorphism $\gamma_{*}\colon\,[\Sigma^{k}\OP^{2},\,O]\rightarrow\,[\Sigma^{k}\OP^{2},\,G]$. Since $[\Sigma^k \OP^2, BO]\cong \mathbb Z^2$ by Lemma~\ref{lem:KO-OP2} and $[\Sigma^k \OP^2, BG]$ is finite, the kernel of $[\Sigma^{k}\OP^{2},\,J]\,\colon [\Sigma^k \OP^2, BO]\to [\Sigma^k \OP^2, BG]$ is a finite-index subgroup of $\mathbb Z^2$. Hence $\operatorname{Ker}[\Sigma^k \OP^2,J]$ is a free abelian group of rank~$2$. This implies that the short exact sequence~\eqref{eq:N-SES} splits.
\end{proof}

\begin{proposition}\label{kerk1}
    For $k\equiv 1\pmod{8}$, we have $\operatorname{Ker}[\Sigma^{k}\OP^2, J]=\Z/2\{\alpha^{2}_{k}\}$.
\end{proposition}
\begin{proof}
For a function $e$ assigning to each integer $\ell$ a non-negative
integer $e(\ell)$, let $W(e,\Sigma^{k}\OP^{2})\subseteq
\widetilde{KO}^{0}(\Sigma^{k}\OP^{2})$ denote the subgroup generated by
the elements $\ell^{e(\ell)}\bigl(\Psi^{\ell}x-x\bigr)$, where
$\ell\in\Z$ and $x\in\widetilde{KO}^{0}(\Sigma^{k}\OP^{2})$
\cite[p.~208--209]{Ada65b}. By \cite[Proposition~3.1,
Theorem~6.1]{Ada65a} and \cite[Theorem~1.1]{Ada65b}, together with the
Adams conjecture \cite[Conjecture~(1.2)]{Ada63}, proved by
Quillen~\cite{Qui71},
\[
\operatorname{Ker}[\Sigma^{k}\OP^{2},J]
=\bigcap_{e}W(e,\Sigma^{k}\OP^{2}).
\]
Since $\widetilde{KO}^{0}(\Sigma^{k}\OP^{2})\cong\Z/2\oplus\Z/2$ by
Lemma~\ref{1mod8}, the generators with odd $\ell$ reduce to
$\Psi^{\ell}x-x$, while those with even $\ell$ vanish once
$e(\ell)\geq 1$. Hence the intersection equals the subgroup generated
by the elements $\Psi^{j}x-x$ with $j$ odd, which by
Proposition~\ref{prop:adams-ops-1mod8} is $\Z/2\{\alpha_{k}^{2}\}$,
since $\Psi^{j}\alpha_{k}-\alpha_{k}=\delta(j)\,\alpha_{k}^{2}$ with
$\delta(j)=1$ for $j\equiv\pm3\pmod 8$, and
$\Psi^{j}\alpha_{k}^{2}=\alpha_{k}^{2}$. This completes the proof.
\end{proof}
\begin{lemma}\label{cok}
 For $k\geq 1$, we have   
 \[T_{k}:=\operatorname{Coker}[\Sigma^{k}\OP^2,\Omega J]=\begin{cases}
        [\Sigma^{k}\OP^2, G] & \text{for }k\equiv 4\pmod8,\\
        [\Sigma^{k}\OP^2, G]/{\Z/2} & \text{for }k\equiv 0\pmod{8},
    \end{cases}\]
    where for $k\equiv0\pmod8$ the subgroup $\Z/2$ is the image of 
$[\Sigma^{k}\OP^{2},\Omega J]\colon
[\Sigma^{k}\OP^{2},O]\rightarrow[\Sigma^{k}\OP^{2},G]$.
\end{lemma}
\begin{proof}
Suppose $k\equiv 4\pmod 8$. Since $\pi_{17+k}(BO)$ and
$\pi_{9+k}(BO)$ vanish by~\eqref{eq:KO-pt-groups}, we obtain from the long exact sequence~\eqref{BO} that $[\Sigma^{k}\OP^{2},O] \cong [\Sigma^{k+1}\OP^{2},BO] = 0$. This implies that $\operatorname{Coker}[\Sigma^{k}\OP^2,\Omega J]\cong [\Sigma^k \OP^2, G]$.

For $k\equiv0\pmod8$ the map $[\Sigma^{k}\OP^{2},\Omega J]$ is adjoint to
$[\Sigma^{k+1}\OP^{2},J]$, whose domain is
$\Z/2\{\alpha_{k+1}\}\oplus\Z/2\{\alpha^{2}_{k+1}\}$ by
Lemma~\ref{lem:module-basis}(b) and whose kernel is
$\Z/2\{\alpha^{2}_{k+1}\}$ by Proposition~\ref{kerk1}; hence its image has
order $2$, and the cokernel is $[\Sigma^{k}\OP^{2},G]\big/(\Z/2)$.
\end{proof}

\begin{theorem}\label{thm:normal-inv-decomp}
For every $k\geq 1$ with $k\equiv 0\pmod{4}$, the group
$\Ndiff_\partial(\OP^2 \times \mathbb{D}^k)$ is finitely generated of rank $3$, and under the splitting~\eqref{eq:N-PT-splitting}, whose torsion subgroup $\operatorname{Im}\overline{\gamma_{*}}\oplus
\operatorname{Im}\eta^{\diff}_{\mathbb{S}^{k}}$ is isomorphic to $T_{k}\oplus\Theta_{k}$, where $T_{k}$ is the finite
abelian group computed in Lemma~\ref{cok}. Consequently,
\[
\Ndiff_\partial(\OP^2 \times \mathbb{D}^k)
\;\cong\;
\Z^{3} \oplus T_{k} \oplus \Theta_{k}.
\]
\end{theorem}

\begin{proof}
By the identification~\eqref{eq:N-PT-splitting}, it suffices to compute the two summands $\pi_{k}(G/O)$ and $[\Sigma^{k}\OP^{2},G/O]$. 

For the first summand, let $k\geq 5$. As $k\equiv 0\pmod 4$, we have
$\pi_{k}(G/O)\cong\Z\oplus\Theta_{k}$, where
$\Theta_{k}$ is embedded in $\pi_{k}(G/O)$ as its torsion subgroup by the
Kervaire--Milnor homomorphism $\eta^{\diff}_{\mathbb{S}^{k}}$; see \cite{andrew}. The same
isomorphism holds for $k=4$, since $\pi_{4}(G/O)\cong\Z$ \cite{Sullivan} and $\Theta_{4}=0$~\cite{homotopysphere}.

For the second summand, Lemmas~\ref{lem:normal-splitting} and~\ref{cok} yield
\[
[\Sigma^k\OP^2, G/O]
\;\cong\;
\Z^{2}\oplus\operatorname{Im}\overline{\gamma_{*}}
\;\cong\;
\Z^{2} \oplus T_{k}.
\]
The result now follows from the splitting~\eqref{eq:N-PT-splitting}.
\end{proof}
\subsection{\texorpdfstring{A basis of the free part of the group $\Ndiff_{\partial}(\OP^2\times\mathbb{D}^{k})$}{A basis of the free part of the normal invariants}}\label{subsec:free-basis}
Theorem~\ref{thm:normal-inv-decomp} identifies
$\Ndiff_{\partial}(\OP^{2}\times\mathbb{D}^{k})_{\mathrm{free}}$
abstractly with $\Z^{3}$, but evaluating the surgery obstruction
requires an explicit description of its generators. In this
subsection we exhibit a basis of
$\operatorname{Ker}[\Sigma^{k}\OP^{2},J]$
(Theorem~\ref{prop:explicit-gens}), which together with a generator
of $\operatorname{Ker}\pi_{k}(J)$ yields the $\Z$-basis
$\{\widetilde{\xi_{\mathbb{S}^{k}}},\widetilde{\xi_{1,k}},
\widetilde{\xi_{2,k}}\}$ used throughout
Section~\ref{sec:obstr-free} (Corollary~\ref{cor:free-basis}).

For a finite CW complex $X$, let $\sh\colon
\widetilde{KO}^{0}(X)\rightarrow
1+\prod_{s>0}\widetilde{H}^{4s}(X;\Q)$ denote the characteristic class associated with the multiplicative
sequence having characteristic power series $\frac{\sinh(\sqrt{z}/2)}{\sqrt{z}/2}$ \cite[p.~154]{Ada65a}. By \cite[Proposition~5.2]{Ada65a}, for every
$\eta\in\widetilde{KO}^{0}(X)$,
\begin{equation}\label{shcharacter}
\log\sh(\eta)
=
\sum_{s\ge1}\frac{a_{2s}}{2}\,\ph_{2s}(\eta),
\end{equation}
where $\ph_r$ denotes the component of the Pontryagin character in
$H^{2r}(X;\Q)$, $\log(1+x)$ is defined by its usual power series, and
the coefficients $a_{2s}$ (denoted $\alpha_{2s}$ in
\cite[p.~154]{Ada65a}) are determined by $\log\!\left(\tfrac{\sinh(x/2)}{x/2}\right)
=
\sum_{s\ge1}a_{2s}\tfrac{x^{2s}}{(2s)!}$.


Since \(\widetilde{KO}^{0}(\Sigma^{k}\OP^{2})\) is torsion free for
\(k\equiv 0 \pmod{4}\) by Lemma~\ref{lem:module-basis}(a), it follows
from \cite[Corollary~2.7]{feder} that
\begin{equation}\label{kernel_description}
\operatorname{Ker}[\Sigma^{k}\OP^{2},J]
=
\{\xi\in \widetilde{KO}^{0}(\Sigma^{k}\OP^{2}) \mid
\sh(\xi)=\ph(1+\zeta)\ \text{for some}\
\zeta\in \widetilde{KO}^{0}(\Sigma^{k}\OP^{2})\}.
\end{equation}
Combining~\eqref{kernel_description} with~\eqref{shcharacter}, and
noting that $\log$ is invertible on
$1+\prod\limits_{s>0}\widetilde{H}^{4s}(\Sigma^{k}\OP^{2};\Q)$, we find that
$\xi\in\operatorname{Ker}[\Sigma^{k}\OP^{2},J]$ if and only if
\begin{equation}\label{eq:sh-ph-log}
\sum_{s \geq 1} \frac{a_{2s}}{2}\,\ph_{2s}(\xi)
= \log\ph(1+\zeta)
= \sum_{i \geq 1} \frac{(-1)^{i-1}}{i}\,\ph(\zeta)^{i},
\end{equation}
for some $\zeta\in\widetilde{KO}^{0}(\Sigma^{k}\OP^{2})$. Since
$\Sigma^{k}\OP^{2}$ is a suspension, all cup products of
positive-degree classes vanish, and hence~\eqref{eq:sh-ph-log} reduces to
\begin{equation}\label{eq:sh-ph-linear}
   \frac{a_{2s}}{2}\,\ph_{2s}(\xi)
   = \ph_{2s}(\zeta),
   \quad \text{for each } s \geq 1.
\end{equation}
Thus $\xi\in\operatorname{Ker}[\Sigma^{k}\OP^{2},J]$ if and only if~\eqref{eq:sh-ph-linear} admits a solution
$\zeta\in\widetilde{KO}^{0}(\Sigma^{k}\OP^{2})$.
\begin{theorem}\label{prop:explicit-gens}
Let $k=4m$ with $m\geq1$, and for $s\geq1$ let $j_{s}$ denote the denominator
of $B_{s}/(4s)$ and $n_{s}$ its numerator, so that $\gcd(n_{s},j_{s})=1$.
Then $\operatorname{Ker}[\Sigma^{k}\OP^{2},J]$ is free abelian of rank $2$,
with a $\Z$-basis given, in the notation of Lemma~\ref{lem:KO-OP2}, by
\begin{equation}\label{kernel_generator}
\xi_{1,k}=j_{m+2}\,\alpha_{k}+c_{m}\,\alpha_{k}^{2},
\qquad
\xi_{2,k}=j_{m+4}\,\alpha_{k}^{2},
\end{equation}
where $c_{m}$ is the unique integer with $0\leq c_{m}<j_{m+4}$ satisfying
\begin{equation}\label{eq:cm-congruence}
n_{m+4}\,c_{m}\equiv
-\frac{n_{m+4}\,j_{m+2}-n_{m+2}\,j_{m+4}}{240}
\pmod{j_{m+4}}.
\end{equation}
\end{theorem}
\begin{proof}
In view of Lemma~\ref{integrality}, the right-hand side of the congruence~\eqref{eq:cm-congruence} is an integer. Since $\gcd(n_{m+4}, j_{m+4}) = 1$, it follows that there exists a unique integer $c_m$ with $0 \le c_m < j_{m+4}$ satisfying the congruence~\eqref{eq:cm-congruence}.

Recall that by equation~\eqref{eqn: ch at beta}, the class
$\beta=c(\alpha)\in\widetilde{KU}^{0}(\OP^{2})$ satisfies
$\ch(\beta)=u+\tfrac{1}{240}u^{2}$. Since $c$ is a ring
homomorphism, $c(\alpha^{2})=\beta^{2}$, and hence
\begin{equation}\label{eq:ch-c-gamma}
\ch(c(\alpha^{2}))
=\Bigl(u+\frac{1}{240}\,u^{2}\Bigr)^{2}
=u^{2},
\end{equation}
as $u^{3}=0$ in $H^{*}(\OP^{2};\Q)$.
The complexification of the Bott class
$\beta_{k}\in KO^{-k}(\mathrm{pt})$ is
$c(\beta_{k})=\kappa_{m}\,b_{U}^{2m}\in KU^{-k}(\mathrm{pt})\cong\Z$,
where $b_{U}\in KU^{-2}(\mathrm{pt})$ is the complex Bott element \cite[p.~62]{spingeo} and
$\kappa_{m}$ is as in~\eqref{eq:kappa-def}. Recall from
Lemma~\ref{lem:module-basis} that
$\widetilde{KO}^{0}(\Sigma^{k}\OP^{2})$ is freely generated by
$\alpha_{k}=\beta_{k}\cdot\alpha$ and
$\alpha_{k}^{2}=\beta_{k}\cdot\alpha^{2}$. Since the Chern character is
multiplicative with respect to external products and
$\ch(b_{U}^{2m})=\iota_{k}$, by equations~\eqref{eqn: ch at beta} and~\eqref{eq:ch-c-gamma}, we obtain
\begin{align}
\ph(\alpha_{k})
&=\ch\bigl(c(\beta_{k})\cdot c(\alpha)\bigr)
=\kappa_{m}\,\iota_{k}\wedge\bigl(u+\tfrac{1}{240}\,u^{2}\bigr)
=\kappa_{m}\bigl(\sigma^{k}u+\tfrac{1}{240}\,\sigma^{k}u^{2}\bigr),
\label{eq:ph-alpha-k}\\[2pt]
\ph(\alpha_{k}^{2})
&=\ch\bigl(c(\beta_{k})\cdot c(\alpha^{2})\bigr)
=\kappa_{m}\,\iota_{k}\wedge u^{2}
=\kappa_{m}\,\sigma^{k}u^{2}.
\label{eq:ph-alpha2-k}
\end{align}

Now, let $\xi=x_{1}\alpha_{k}+x_{2}\alpha_{k}^{2}$ and
$\zeta=y_{1}\alpha_{k}+y_{2}\alpha_{k}^{2}$ with
$x_{1},x_{2},y_{1},y_{2}\in\Z$. By~\eqref{eq:ph-alpha-k} and
\eqref{eq:ph-alpha2-k}, we have
\begin{align*}
\ph(\xi)=\kappa_{m}\Bigl[x_{1}\,\sigma^{k}u
+\bigl(\tfrac{x_{1}}{240}+x_{2}\bigr)\sigma^{k}u^{2}\Bigr],\quad
\ph(\zeta)=\kappa_{m}\Bigl[y_{1}\,\sigma^{k}u
+\bigl(\tfrac{y_{1}}{240}+y_{2}\bigr)\sigma^{k}u^{2}\Bigr].
\end{align*}
Since $\widetilde{H}^{*}(\Sigma^{k}\OP^{2};\Q)$ is concentrated in
degrees $k+8=4(m+2)$ and $k+16=4(m+4)$, the component $\ph_{2s}$
vanishes unless $s\in\{m+2,\,m+4\}$. Hence, by equation
\eqref{eq:sh-ph-linear}, the relation $\sh(\xi)=\ph(1+\zeta)$ holds if
and only if, after cancelling the common factor $\kappa_{m}$,
\begin{equation}\label{eq:first-scalar}
\frac{a_{4+2m}}{2}\,x_{1}=y_{1}
\end{equation}
and
\begin{equation}\label{eq:second-scalar}
\frac{a_{8+2m}}{2}\Bigl(\frac{x_{1}}{240}+x_{2}\Bigr)
=\frac{y_{1}}{240}+y_{2}.
\end{equation}
Now we observe from \cite[p.~138]{Ada65a} that
$\frac{a_{2s}}{2}=(-1)^{s-1}\frac{B_{s}}{4s}
=(-1)^{s-1}\frac{n_{s}}{j_{s}}$. Setting
$\epsilon_{m}:=(-1)^{m+1}$ and noting that
$(-1)^{(m+2)-1}=(-1)^{(m+4)-1}=\epsilon_{m}$, we can rewrite equations~\eqref{eq:first-scalar} and~\eqref{eq:second-scalar} as
\begin{align}
\epsilon_{m}\,\frac{n_{m+2}}{j_{m+2}}\,x_{1}&=y_{1},
\label{eq:row1-B}\\[2pt]
\epsilon_{m}\,\frac{n_{m+4}}{j_{m+4}}
\Bigl(\frac{x_{1}}{240}+x_{2}\Bigr)&=\frac{y_{1}}{240}+y_{2}.
\label{eq:row2-B}
\end{align}
In view of~\eqref{kernel_description}, we conclude that
$\xi\in\operatorname{Ker}[\Sigma^{k}\OP^{2},J]$ if and only if there
exist $y_{1},y_{2}\in\Z$ satisfying equations~\eqref{eq:row1-B} and
\eqref{eq:row2-B}. Write 
$\Delta:=\tfrac{n_{m+4}\,j_{m+2}-n_{m+2}\,j_{m+4}}{240}$, with $\Delta \in \mathbb{Z}$ by Lemma~\ref{integrality}.

We first show that the classes $\xi_{1,k}$, $\xi_{2,k}$ of
\eqref{kernel_generator} lie in
$\operatorname{Ker}[\Sigma^{k}\OP^{2},J]$. For $\xi_{1,k}$, that is,
$(x_{1},x_{2})=(j_{m+2},c_{m})$, the integers
$y_{1}=\epsilon_{m}\,n_{m+2}$ and $y_{2}=\epsilon_{m}\,w_{m}$, where
\[
w_{m}
:=\frac{n_{m+4}}{j_{m+4}}\Bigl(\frac{j_{m+2}}{240}+c_{m}\Bigr)
-\frac{n_{m+2}}{240}
=\frac{n_{m+4}\,c_{m}+\Delta}{j_{m+4}},
\]
solve equations~\eqref{eq:row1-B} and~\eqref{eq:row2-B}; here $w_{m}\in\Z$
because $n_{m+4}\,c_{m}+\Delta\equiv 0\pmod{j_{m+4}}$ by
the congruence~\eqref{eq:cm-congruence}. For $\xi_{2,k}$, that is,
$(x_{1},x_{2})=(0,j_{m+4})$, the integers $y_{1}=0$ and
$y_{2}=\epsilon_{m}\,n_{m+4}$ solve~ equations\eqref{eq:row1-B} and
\eqref{eq:row2-B}. Moreover, $\xi_{1,k}$ and $\xi_{2,k}$ are
$\Z$-linearly independent, as their coefficient matrix with respect to
the basis $\{\alpha_{k},\alpha_{k}^{2}\}$ is
\[
\begin{pmatrix} j_{m+2} & 0 \\ c_{m} & j_{m+4}\end{pmatrix},
\]
has determinant $j_{m+2}\,j_{m+4}>0$.

It remains to show that $\xi_{1,k}$ and $\xi_{2,k}$ generate
$\operatorname{Ker}[\Sigma^{k}\OP^{2},J]$. Suppose
$\xi=x_{1}\alpha_{k}+x_{2}\alpha_{k}^{2}
\in\operatorname{Ker}[\Sigma^{k}\OP^{2},J]$, and let
$y_{1},y_{2}\in\Z$ satisfy~equations \eqref{eq:row1-B} and~\eqref{eq:row2-B}.
Since $\gcd(n_{m+2},j_{m+2})=1$ and $y_{1}\in\Z$, equation
\eqref{eq:row1-B} yields $j_{m+2}\mid x_{1}$. Write
$x_{1}=j_{m+2}\,t$ with $t\in\Z$; then
$y_{1}=\epsilon_{m}\,n_{m+2}\,t$. Then equation~\eqref{eq:row2-B} yields
\[
\epsilon_{m}\,y_{2}
=\frac{n_{m+4}}{j_{m+4}}\Bigl(\frac{j_{m+2}\,t}{240}+x_{2}\Bigr)
-\frac{n_{m+2}\,t}{240}
=\frac{n_{m+4}\,x_{2}+t\,\Delta}{j_{m+4}}.
\]
Since $\epsilon_{m}\,y_{2}\in\Z$, it follows that
$n_{m+4}\,x_{2}\equiv-t\,\Delta\pmod{j_{m+4}}$. Comparing with
the congruence~\eqref{eq:cm-congruence}, we obtain
$n_{m+4}\,x_{2}\equiv t\,n_{m+4}\,c_{m}\pmod{j_{m+4}}$, and since
$\gcd(n_{m+4},j_{m+4})=1$, we obtain
$x_{2}\equiv t\,c_{m}\pmod{j_{m+4}}$. Hence
$x_{2}=t\,c_{m}+j_{m+4}\,q$ for some $q\in\Z$, and therefore
\[
\xi=(j_{m+2}\,t)\,\alpha_{k}+(t\,c_{m}+j_{m+4}\,q)\,\alpha_{k}^{2}
=t\,\xi_{1,k}+q\,\xi_{2,k},
\]
which completes the proof.
\end{proof}

Set $\xi_{\mathbb{S}^{k}}:=j_{m}\,\beta_{k}\in\pi_{k}(BO)$; by
\cite[Theorem~3]{groth}, together with \cite[Theorem~1.1]{Qui71},
which settles the ambiguity in the case
$k\equiv 0\pmod 8$, the class $\xi_{\mathbb{S}^{k}}$ generates
$\operatorname{Ker}\pi_{k}(J)\cong\Z$.
\begin{corollary}\label{cor:free-basis}
Let $k=4m$ with $m\geq 1$, and let $\xi_{1,k}$, $\xi_{2,k}$ be as in
Equation~\eqref{kernel_generator}, so that $\{\xi_{1,k},\xi_{2,k}\}$ is
a basis of $\operatorname{Ker}[\Sigma^{k}\OP^{2},J]$ by
Theorem~\ref{prop:explicit-gens}. Then there exist classes
$\widetilde{\xi_{1,k}},\widetilde{\xi_{2,k}}\in[\Sigma^{k}\OP^{2},G/O]$
and $\widetilde{\xi_{\mathbb{S}^{k}}}\in\pi_{k}(G/O)$, both groups
regarded as subgroups of $[X_{k},G/O]$
via the splitting~\eqref{eq:N-PT-splitting}, with
\[
r_{*}(\widetilde{\xi_{1,k}})=\xi_{1,k},\qquad
r_{*}(\widetilde{\xi_{2,k}})=\xi_{2,k},\qquad
r_{*}(\widetilde{\xi_{\mathbb{S}^{k}}})=\xi_{\mathbb{S}^{k}},
\]
where $r\colon G/O\to BO$ denotes the canonical map. For any such choice the set
$\{\widetilde{\xi_{\mathbb{S}^{k}}},\widetilde{\xi_{1,k}},
\widetilde{\xi_{2,k}}\}$ forms a basis
of $\Ndiff_{\partial}(\OP^{2}\times\mathbb{D}^{k})_{\rm free}\cong\Z^{3}$. 
\end{corollary}
\begin{proof}
Following splitting~\eqref{eq:N-PT-splitting}, the free part of the relative normal invariants decomposes as
\[
\Ndiff_{\partial}(\OP^{2}\times\mathbb{D}^{k})_{\mathrm{free}}
\cong \operatorname{Ker}\pi_{k}(J) \oplus \operatorname{Ker}[\Sigma^{k}\OP^{2},J].
\]
Note that the right-hand side of this isomorphism admits a $\Z$-basis $\{\xi_{\mathbb{S}^{k}}, \xi_{1,k}, \xi_{2,k}\}$. Since $[X, G]$ is a finite group for any finite CW complex $X$, the canonical map $r_{*}$ induces an isomorphism on the torsion-free parts. Therefore, the chosen lifts $\{\widetilde{\xi_{\mathbb{S}^{k}}}, \widetilde{\xi_{1,k}}, \widetilde{\xi_{2,k}}\}$ necessarily form a $\Z$-basis for $\Ndiff_{\partial}(\OP^{2}\times\mathbb{D}^{k})_{\mathrm{free}}$.
\end{proof}

Let $\{\widetilde{\xi_{\mathbb{S}^{k}}},\widetilde{\xi_{1,k}},
\widetilde{\xi_{2,k}}\}$ denote lifts of the chosen generators of
$\Ndiff_{\partial}(\OP^{2}\times\mathbb{D}^{k})_{\mathrm{free}}$. The subsequent computations are independent of this choice, since any
two lifts differ by a torsion element, which is annihilated by
$\sigma^{\diff}_{16+k}$.
 \section{The Surgery Obstruction} \label{sec:obstr-free}
 
In this section we determine the surgery obstruction map
$\sigma^{\diff}_{16+k}$ for $\OP^{2}\times\mathbb{D}^{k}$, where
$k\equiv 0\pmod 4$ and $k\geq 1$. We compute
its values on the $\Z$-basis of
$\Ndiff_{\partial}(\OP^{2}\times\mathbb{D}^{k})_{\mathrm{free}}$
constructed in Corollary~\ref{cor:free-basis}
(Theorem~\ref{thm:sigma-on-Sigma-kOP2}), thereby establishing the
formulas of Theorem~\ref{thm:image-obstruction}, and then exhibit a
free basis of the kernel of $\sigma^{\diff}_{16+k}$ on the free part
(Proposition~\ref{prop:kernel-sigma}).

Recall that an element of
$\Ndiff_{\partial}(\OP^{2}\times\mathbb{D}^{k})$ is represented by
a normal bordism class $[(M,\partial M),(f,\partial f)]$, where
$(f,\partial f)\colon(M,\partial M)\to
(\OP^{2}\times\mathbb{D}^{k},\,\OP^{2}\times\mathbb{S}^{k-1})$ is a
degree-one normal map restricting to a diffeomorphism on the boundary
\cite[Chapter~7]{LM24}. The surgery obstruction map $\sigma^{\diff}_{16+k}\colon
\Ndiff_{\partial}(\OP^{2}\times\mathbb{D}^{k})
\rightarrow L_{16+k}(\Z)\cong\Z$ is given by (see \cite[p.~10]{georg})
\begin{equation}\label{eq:browder-additivity}
\sigma^{\diff}_{16+k}(f,\partial f)
=\tfrac{1}{8}\bigl(\operatorname{sign}(\widehat{M})
-\operatorname{sign}(\OP^{2}\times\mathbb{S}^{k})\bigr)
=\tfrac{1}{8}\operatorname{sign}(\widehat{M}),
\end{equation}
where $\widehat{M}=M\cup_{\partial M}(\OP^{2}\times\mathbb{D}^{k})$,
and the second equality holds since
$\operatorname{sign}(\OP^{2}\times\mathbb{S}^{k})
=\operatorname{sign}(\OP^{2})\cdot\operatorname{sign}(\mathbb{S}^{k})
=0$.

In view of the isomorphism
$\Ndiff_{\partial}(\OP^{2}\times\mathbb{D}^{k})\cong[X_{k},G/O]$, an element of
$\Ndiff_{\partial}(\OP^{2}\times\mathbb{D}^{k})$ may be
represented by a pair $[\xi,t]$, where $\xi$ is a rank-$N$ vector
bundle over $X_{k}$, for some large $N$, and
$t\colon S(\xi)\to X_{k}\times\mathbb{S}^{N-1}$ is a fibre homotopy
trivialization of the associated sphere bundle. Under this
isomorphism, $[\xi,t]$ corresponds to the normal bordism class of a
degree-one normal map
\begin{equation}\label{normalsquare}
\begin{tikzcd}
\nu_{M} \arrow[r, "\bar{f}"] \arrow[d]
& \nu_{\OP^{2}\times\mathbb{D}^{k}}\oplus \xi \arrow[d] \\
(M,\partial M) \arrow[r, "{(f,\,\partial f)}"']
& (\OP^{2}\times\mathbb{D}^{k},\,\OP^{2}\times\mathbb{S}^{k-1})
\end{tikzcd}
\end{equation}
such that $\partial f$ is a diffeomorphism, covered by an isomorphism
of bundles. Here $\nu_{(-)}$ denotes the stable normal bundle. 
Gluing the normal map
$(\mathrm{id}_{\OP^{2}\times\mathbb{D}^{k}},
\mathrm{id}_{\OP^{2}\times\mathbb{S}^{k-1}})$ along the boundary, we
obtain a degree-one normal map
$\widehat{f}\colon\widehat{M}\to\OP^{2}\times\mathbb{S}^{k}$ covered by a
bundle map
$\nu_{\widehat{M}}\to\nu_{\OP^{2}\times\mathbb{S}^{k}}\oplus\pi^{*}\xi$,
where $\pi\colon\OP^{2}\times\mathbb{S}^{k}\to X_{k}$ collapses $\OP^2\times\{pt\}$ to a point. Since
$\mathcal{L}(T\widehat{M})\cdot\mathcal{L}(\nu_{\widehat{M}})=1$, the
multiplicativity and naturality of the total Hirzebruch $L$-class,
together with the fact that $\nu_{\OP^{2}\times\mathbb{S}^{k}}$ is a
stable inverse of $T(\OP^{2}\times\mathbb{S}^{k})$, yield \[\mathcal{L}(\widehat{M})
=\widehat{f}^{*}\Bigl(\mathcal{L}(\OP^{2}\times\mathbb{S}^{k})
\cdot\pi^{*}\bigl(\mathcal{L}(\xi)^{-1}\bigr)\Bigr).\]
Hence, by the Hirzebruch signature theorem and the fact that $\widehat{f}$
has degree one,
\begin{equation}\label{eq:sign-hatM}
\operatorname{sign}(\widehat{M})
=\bigl\langle\mathcal{L}(\widehat{M}),[\widehat{M}]\bigr\rangle
=\Bigl\langle\mathcal{L}(\OP^{2}\times\mathbb{S}^{k})
\cdot\pi^{*}\bigl(\mathcal{L}(\xi)^{-1}\bigr),
[\OP^{2}\times\mathbb{S}^{k}]\Bigr\rangle.
\end{equation}
Since $T(\OP^{2}\times\mathbb{S}^{k})\cong
\mathrm{proj}_{1}^{*}\,T\OP^{2}\oplus\mathrm{proj}_{2}^{*}\,T\mathbb{S}^{k}$,
where $\mathrm{proj}_{i}$ denotes the projection onto the $i$-th
factor, and $\mathcal{L}(T\mathbb{S}^{k})=1$ as $T\mathbb{S}^{k}$ is
stably trivial,~\eqref{eq:sign-hatM} reduces to
\begin{equation}\label{eq:sign-reduced}
\operatorname{sign}(\widehat{M})
=\Bigl\langle\mathrm{proj}_{1}^{*}\,\mathcal{L}(\OP^{2})
\cdot\pi^{*}\bigl(\mathcal{L}(\xi)^{-1}\bigr),
[\OP^{2}\times\mathbb{S}^{k}]\Bigr\rangle.
\end{equation}
Since $X_{k}\simeq\Sigma^{k}\OP^{2}\vee\mathbb{S}^{k}$, the
cup product of any two classes of positive degree in
$\widetilde{H}^{*}(X_{k};\Q)$ vanishes. Writing
$\mathcal{L}(\xi)=1+x$ with $x$ of positive degree, the identity
$\mathcal{L}(\xi)\,\mathcal{L}(\xi)^{-1}=1$ forces
$\mathcal{L}(\xi)^{-1}=1-x=2-\mathcal{L}(\xi)$. Combining~\eqref{eq:sign-reduced} with~\eqref{eq:browder-additivity}, we obtain the following lemma.
\begin{lemma}\label{lem:sigma-formula}
Let $k\equiv 0 \pmod 4$ with $k\ge 1$. 
Then the relative surgery obstruction map
$\sigma^{\diff}_{16+k}\colon [X_{k},\,G/O] \rightarrow L_{16+k}(\mathbb Z)\cong \mathbb Z$ is given by
\begin{equation*} 
\sigma^{\diff}_{16+k}([\xi,t])
=
-\frac18
\Bigl\langle
\pi^*\!\bigl(\mathcal L(\xi)-2\bigr)\,
\mathrm{proj}_1^*\mathcal L(\OP^2),
[\OP^2\times S^k]
\Bigr\rangle .
\end{equation*}
Consequently, $\sigma^{\diff}_{16+k}$ depends only on the image of
$[X_{k},G/O]
\xrightarrow{\,r_*\,}
[X_{k},\,BO]
\xrightarrow{\,\pi^*\,}
[\OP^2\times \mathbb{S}^k,BO]$,
where $r\colon G/O\to BO$ denotes the canonical map.
\end{lemma}

\subsection{Evaluation of Surgery Obstruction on \texorpdfstring{$\Ndiff_{\partial}(\OP^{2}\times \mathbb{D}^{k})_{\mathrm{free}}$}{}}
In this subsection we evaluate the surgery obstruction on the
$\Z$-basis of
$\Ndiff_{\partial}(\OP^{2}\times\mathbb{D}^{k})_{\mathrm{free}}$
constructed in Corollary~\ref{cor:free-basis}, obtaining the integers
$S$, $A$ and $B$ of Theorem~\ref{thm:image-obstruction}
(Theorem~\ref{thm:sigma-on-Sigma-kOP2}). We do so by first expressing
the surgery obstruction of an arbitrary normal invariant in
$[\Sigma^{k}\OP^{2},G/O]$ in terms of the Pontryagin character
coordinates $a_{\xi}$, $b_{\xi}$ of the underlying stable bundle
$\xi$ (Theorem~\ref{thm:sigma-general-xi}); this general formula is
needed again in Section~\ref{sec:applications}.
\begin{lemma}\label{lem:L-eq-D-ph}
Let $X$ be a topological space such that
$a \smile b = 0$ for all $a,b \in H^{>0}(X;\Q)$,
and let $\eta$ be a stable real vector bundle over $X$. For each $s \ge 1$, set
\[
D_s := \frac{(-1)^{s+1}2^{2s}(2^{2s-1}-1)B_s}{2s},
\]
where $B_s$ denotes the Bernoulli numbers. Then $L_s(\eta)=D_s\,\mathrm{ph}_{2s}(\eta)$
in $H^{4s}(X;\Q)$, where $L_s(\eta)$ and $\mathrm{ph}_{2s}(\eta)$ denote
the degree-$4s$ components of the Hirzebruch $L$-class and the Pontryagin
character $\mathrm{ph}(\eta):=\mathrm{ch}(\eta\otimes_{\R}\mathbb{C})$, respectively.
\end{lemma}
\begin{proof}
By Newton’s identities \cite[Page~92]{Hir66} and the assumption on cup-product vanishing, all mixed products in the Chern character expansion vanish. Hence the degree-$4s$ component satisfies
$\mathrm{ch}_{2s}(\eta\otimes\mathbb{C})
=
-\tfrac{1}{(2s-1)!}\,c_{2s}(\eta\otimes \mathbb{C})$.
It follows from $p_s(\eta)=(-1)^s c_{2s}(\eta\otimes \mathbb{C})$, that
\begin{equation}\label{phchar}
\mathrm{ph}_{2s}(\eta)
=
\mathrm{ch}_{2s}(\eta\otimes \mathbb{C})
=
\frac{(-1)^{s+1}}{(2s-1)!}\,p_s(\eta).
\end{equation}
Under the assumption that all cup products of positive-degree classes vanish, all decomposable terms in the Hirzebruch $L$-polynomial vanish, so only the top Pontryagin class $p_s(\eta)$ contributes to $L_s(\eta)$. The stated formula then follows from the standard expression for the $L$-polynomial (see \cite[Problem 19-C]{milnor}).
\end{proof}
\begin{theorem}\label{thm:sigma-general-xi}
Let $k=4m$ with $m\geq 1$, and let $\xi \in [\Sigma^k\OP^2,BO] \subset [X_{k},BO]$.
Let $a_\xi,b_\xi\in \Q$ be determined by
\begin{equation}\label{eq:ph-coords}
\widetilde{\ph}(\xi)
=
a_\xi\,\sigma^k u
+
b_\xi\,\sigma^k u^2
\in
\widetilde H^*(\Sigma^k\OP^2;\Q),
\end{equation}
where
$\widetilde{\ph}$ 
 denotes the reduced
Pontryagin character.
Then, for any lift $[\widetilde{\xi}]$ of $\xi$ to
$[\Sigma^k\OP^2,G/O]$, the relative surgery obstruction is
\begin{equation*}
\sigma^{\diff}_{16+k}([\widetilde{\xi}])
=
- \frac{1}{8}
\left(
\frac{14D_{m+2}}{15}\,a_\xi
+
D_{m+4}\,b_\xi
\right).
\end{equation*}
\end{theorem}
\begin{proof}
By Lemma~\ref{lem:L-eq-D-ph} together with~\eqref{eq:ph-coords}, and
using $\pi^{*}(\sigma^{k}u)=u\times\iota_{k}$ and
$\pi^{*}(\sigma^{k}u^{2})=u^{2}\times\iota_{k}$, we have
$\pi^{*}\bigl(\mathcal{L}(\xi)-2\bigr)
=-1+D_{m+2}\,a_{\xi}\,(u\times\iota_{k})
+D_{m+4}\,b_{\xi}\,(u^{2}\times\iota_{k})$.

Moreover, $p_{2}(\OP^{2})=6u$ and $p_{4}(\OP^{2})=39u^{2}$ by
\cite[Theorem~19.4]{borel}, while $p_{1}(\OP^{2})$ and $p_{3}(\OP^{2})$
vanish for degree reasons; hence \cite[p.~225]{milnor} gives
$\mathcal{L}(\OP^{2})=1+\tfrac{14}{15}\,u+u^{2}$. Since $u^{3}=0$ in
$H^{*}(\OP^{2};\Q)$, we obtain
\begin{align*}
\pi^{*}\bigl(\mathcal{L}(\xi)-2\bigr)\cdot
\mathrm{proj}_{1}^{*}\,\mathcal{L}(\OP^{2})
={}&-\bigl(1+\tfrac{14}{15}(u\times1)+(u^{2}\times1)\bigr)
+D_{m+2}\,a_{\xi}\,(u\times\iota_{k})\\
&+\bigl(\tfrac{14}{15}\,D_{m+2}\,a_{\xi}
+D_{m+4}\,b_{\xi}\bigr)(u^{2}\times\iota_{k}).
\end{align*}
Only the component of degree $16+k$ pairs non-trivially with
$[\OP^{2}\times\mathbb{S}^{k}]$, and
$\bigl\langle u^{2}\times\iota_{k},
[\OP^{2}\times\mathbb{S}^{k}]\bigr\rangle=1$. Now, the result follows from Lemma~\ref{lem:sigma-formula}.
\end{proof}
\begin{theorem}\label{thm:sigma-on-Sigma-kOP2}
Let $k=4m$ with $m\geq 1$.
\begin{itemize}
\item[(a)] Let $\widetilde{\xi_{1,k}}$ and $\widetilde{\xi_{2,k}}$ be be as in Corollary~\ref{cor:free-basis}.
Then the restriction of relative surgery obstruction $\sigma^{\rm DIFF}_{16+k}$ to $[\Sigma^{k}\OP^{2},\,G/O]\subseteq [X_{k}, G/O]$ is given by 
\begin{align}
\sigma^{\diff}_{16+k}(\widetilde{\xi_{1,k}})
   &= - \,\frac{\kappa_m}{8}\!\left[
      \left(\frac{14\,D_{m+2}}{15}
      + \frac{D_{m+4}}{240}\right) j_{m+2}
      + D_{m+4}\,c_m
   \right], \label{eq:sigma-xi1}\\[4pt]
\sigma^{\diff}_{16+k}(\widetilde{\xi_{2,k}})
   &= - \,\frac{\kappa_m}{8}\,D_{m+4}\,j_{m+4}. \label{eq:sigma-xi2}
\end{align}
\item[(b)]  Let $\widetilde{\xi_{\mathbb{S}^{k}}}$ be as in Corollary~\ref{cor:free-basis}. Then the restriction of relative surgery obstruction $\sigma^{\rm DIFF}_{16+k}$ to $\pi_{k}(G/O)\subseteq [X_{k}, G/O]$ is given by 
\begin{equation}\label{eq:sigma-xi-S}
   \sigma^{\diff}_{16+k}(\widetilde\xi_{\mathbb{S}^{k}})
   \;=-\frac{D_{m}j_{m}\kappa_{m}}{8}\;=\;
   (-1)^m\,2^{2m-2}\,(2^{2m-1} - 1)\,n_m\,\kappa_m ,
\end{equation}
\end{itemize}
where $D_{s}$, $n_{m}$ and $\kappa_{m}$ are as in Lemma~\ref{lem:L-eq-D-ph}, Theorem~\ref{prop:explicit-gens}
and~\eqref{eq:kappa-def}, respectively.
\end{theorem}
\begin{proof}
By~\eqref{eq:ph-alpha-k} and~\eqref{eq:ph-alpha2-k}, the
reduced Pontryagin characters of the generators $\xi_{1,k}$, $\xi_{2,k}$
of Theorem~\ref{prop:explicit-gens} are
\[
\widetilde{\ph}(\xi_{1,k})
=\kappa_{m}\Bigl(j_{m+2}\,\sigma^{k}u
+\bigl(\tfrac{j_{m+2}}{240}+c_{m}\bigr)\sigma^{k}u^{2}\Bigr),
\qquad
\widetilde{\ph}(\xi_{2,k})
=\kappa_{m}\,j_{m+4}\,\sigma^{k}u^{2}.
\]
Hence, with respect to the coordinates in~\eqref{eq:ph-coords},
$(a_{\xi_{1,k}},b_{\xi_{1,k}})
=\bigl(\kappa_{m}j_{m+2},\,
\kappa_{m}(\tfrac{j_{m+2}}{240}+c_{m})\bigr)$ and
$(a_{\xi_{2,k}},b_{\xi_{2,k}})=(0,\,\kappa_{m}j_{m+4})$. Part~(a) now
follows from Theorem~\ref{thm:sigma-general-xi}.

For part~(b), since $c(\beta_{k})=\kappa_{m}\,b_{U}^{2m}$ and
$\ch(b_{U}^{2m})=\iota_{k}$, we obtain
$\widetilde{\ph}(\beta_{k})=\kappa_{m}\,\iota_{k}$; hence
the bundle ${\xi}_{\mathbb{S}^{k}}=j_{m}\beta_{k}$ satisfies
\[
\widetilde{\ph}(\xi_{\mathbb{S}^{k}})=j_{m}\,\kappa_{m}\,\iota_{k}.
\]
and
Lemma~\ref{lem:L-eq-D-ph} yields 
\[
\pi^{*}\bigl(\mathcal{L}(\xi_{\mathbb{S}^{k}})-2\bigr)
=-1+D_{m}\,j_{m}\,\kappa_{m}\,(1\times\iota_{k}).\]
Since the coefficient of $u^{2}$ in $\mathcal{L}(\OP^{2})$ equals
$1$, part~(b) now follows from Lemma~\ref{lem:sigma-formula}.
\end{proof}

\subsection{The kernel of the relative surgery obstruction map}
By exactness of~\eqref{eq:surgery-even}, the structure set
$\Sdiff_{\partial}(\OP^{2}\times\mathbb{D}^{k})$ is the kernel of
$\sigma^{\diff}_{16+k}$. We therefore conclude the section by
exhibiting a free basis
$\{\widetilde{\eta_{1,k}},\widetilde{\eta_{2,k}}\}$ of this kernel on the
free part (Proposition~\ref{prop:kernel-sigma}); these classes appear
in Theorem~\ref{thm:main-structure} and in the proof of
Corollary~\ref{cor:infinitely-many}.
\begin{proposition}\label{prop:kernel-sigma}
Let $k=4m$ with $m\geq 1$, and let
$S_{m}=\sigma^{\diff}_{16+k}(\widetilde{\xi_{\mathbb{S}^{k}}})$,
$A_{m}=\sigma^{\diff}_{16+k}(\widetilde{\xi_{1,k}})$ and
$U_{m}=\sigma^{\diff}_{16+k}(\widetilde{\xi_{2,k}})$ be as in
Theorem~\ref{thm:image-obstruction}. Then:
\begin{enumerate}
\item[\textup{(a)}] $\ker\bigl(\sigma^{\diff}_{16+k}\big|_{\mathrm{free}}\bigr)$ is a free abelian group of rank $2$.

\item[\textup{(b)}] Set $d_{m}:=\gcd(A_{m},U_{m})$, so that
$g_{m}=\gcd(S_{m},A_{m},U_{m})=\gcd(S_{m},d_{m})$, 
and fix $\lambda_{m},\mu_{m}\in\Z$ with
$\lambda_{m} A_{m}+\mu_{m} U_{m}=d_{m}$. Then the elements
\[
  \widetilde{\eta_{1,k}}:=\frac{U_{m}}{d_{m}}\,\widetilde{\xi_{1,k}}
        -\frac{A_{m}}{d_{m}}\,\widetilde{\xi_{2,k}},
  \qquad
  \widetilde{\eta_{2,k}}:=\frac{d_{m}}{g_{m}}\,\widetilde{\xi_{\mathbb{S}^{k}}}
         -\frac{S_{m}\lambda_{m}}{g_{m}}\,\widetilde{\xi_{1,k}}
         -\frac{S_{m}\mu_{m}}{g_{m}}\,\widetilde{\xi_{2,k}}
\]
form a $\Z$-basis of $\ker\bigl(\sigma^{\diff}_{16+k}\big|_{\mathrm{free}}\bigr)$.
\end{enumerate}
Hence, under the splitting~\eqref{eq:N-PT-splitting} and in view of
Theorem~\ref{thm:normal-inv-decomp},
\[
\operatorname{Ker}\sigma^{\diff}_{16+k}=
\Z\{\widetilde{\eta_{1,k}}\}\oplus\Z\{\widetilde{\eta_{2,k}}\}
\oplus\operatorname{Im}\overline{\gamma_{*}}
\oplus\operatorname{Im}\eta^{\diff}_{\mathbb{S}^{k}}\;\cong\;\Z^{2}\oplus T_{k}\oplus\Theta_{k},
\]
where the last two summands form the torsion subgroup and the monomorphisms
$\overline{\gamma_{*}}$ and $\eta^{\diff}_{\mathbb{S}^{k}}$ identifying them
with $T_{k}$ and $\Theta_{k}$ respectively.
\end{proposition}
\begin{proof}
\noindent\textit{Proof of \textup{(a)}.} We note from equation~\eqref{eq:sigma-xi2}
that $U_{m}\neq 0$. This gives that the homomorphism
$\sigma^{\diff}_{16+k}\big|_{\mathrm{free}}\colon\Z^3\to\Z$
is nonzero. Hence the image of $\sigma^{\diff}_{16+k}\big|_{\mathrm{free}}$ is
$\gcd(S_{m},A_{m},U_{m})\,\Z$ which is of rank $1$. Consequently,
$\ker\bigl(\sigma^{\diff}_{16+k}\big|_{\mathrm{free}}\bigr)$ is a free abelian
group of rank $2$.

\medskip\noindent
\textit{Proof of \textup{(b)}.} We now verify that the elements
$\widetilde{\eta_{1,k}}$ and $\widetilde{\eta_{2,k}}$ lie in
$\ker\bigl(\sigma^{\diff}_{16+k}\big|_{\mathrm{free}}\bigr)$. Using
$\lambda_{m} A_{m}+\mu_{m} U_{m}=d_{m}$ we obtain
\[
  \sigma^{\diff}_{16+k}\big|_{\mathrm{free}}(\widetilde{\eta_{1,k}})
  = \tfrac{U_{m}}{d_{m}}A_{m} - \tfrac{A_{m}}{d_{m}}U_{m} = 0,
\]
\[
  \sigma^{\diff}_{16+k}\big|_{\mathrm{free}}(\widetilde{\eta_{2,k}})
  = \tfrac{d_{m}}{g_{m}}S_{m}
    - \tfrac{S_{m}}{g_{m}}\bigl(\lambda_{m}A_{m} + \mu_{m}U_{m}\bigr)
  = \tfrac{d_{m}}{g_{m}}S_{m} - \tfrac{S_{m}}{g_{m}}d_{m} = 0,
\]
and hence $\widetilde{\eta_{1,k}},\widetilde{\eta_{2,k}} \in
\ker\bigl(\sigma^{\diff}_{16+k}\big|_{\mathrm{free}}\bigr)$.

It remains to show that $\widetilde{\eta_{1,k}}$ and $\widetilde{\eta_{2,k}}$
generate the kernel. In the basis
$\{\widetilde{\xi_{\mathbb{S}^{k}}},\widetilde{\xi_{1,k}},\widetilde{\xi_{2,k}}\}$ of
$\Ndiff_{\partial}(\OP^2\times\mathbb D^{k})_{\mathrm{free}}\cong\Z^3$
(see Corollary~\ref{cor:free-basis}), the coefficient matrix of
$(\widetilde{\eta_{1,k}},\widetilde{\eta_{2,k}})$ is
\[
  M_{m}=
  \begin{pmatrix}
    0            & d_{m}/g_{m} \\
    U_{m}/d_{m}  & -S_{m}\lambda_{m}/g_{m} \\
    -A_{m}/d_{m} & -S_{m}\mu_{m}/g_{m}
  \end{pmatrix}.
\]
The $2 \times 2$ minors of $M_{m}$ are $-\frac{U_{m}}{g_{m}}$,
$\frac{A_{m}}{g_{m}}$, and $-\frac{S_{m}}{g_{m}}$. Since
$g_{m} = \gcd(S_{m},A_{m},U_{m})$, the greatest common divisor of these minors
is $1$. As the product of the invariant factors of $M_{m}$ equals this greatest
common divisor, and the matrix $M_{m}$ has integer entries, it follows that both
invariant factors are $1$. Consequently, the subgroup
$H_{k} := \langle\widetilde{\eta_{1,k}}, \widetilde{\eta_{2,k}}\rangle$ is a direct
summand of $\Z^3$ of rank $2$, implying that the quotient $\Z^3/H_{k}$ is
torsion-free. By part (a),
$\ker\bigl(\sigma^{\diff}_{16+k}\big|_{\mathrm{free}}\bigr)$ is a free abelian
group of rank $2$. Since
$H_{k} \subseteq \ker\bigl(\sigma^{\diff}_{16+k}\big|_{\mathrm{free}}\bigr)$ and both
groups have the same rank, the quotient
$\ker\bigl(\sigma^{\diff}_{16+k}\big|_{\mathrm{free}}\bigr)/H_{k}$ is a finite group.
However, because this finite quotient naturally embeds as a subgroup of the
torsion-free group $\Z^3/H_{k}$, it must be trivial. We conclude that
$H_{k} = \ker\bigl(\sigma^{\diff}_{16+k}\big|_{\mathrm{free}}\bigr)$.

Finally, since $\sigma^{\diff}_{16+k}$ is a homomorphism into $\mathbb{Z}$, it vanishes on the torsion subgroup of $\Ndiff_{\partial}(\OP^{2}\times\mathbb{D}^{k})$. Therefore,
$\operatorname{Ker}\sigma^{\diff}_{16+k}
=
\operatorname{Ker}\bigl(\sigma^{\diff}_{16+k}\!\mid_{\mathrm{free}}\bigr)
\oplus
\operatorname{Im}\overline{\gamma_{*}}
\oplus
\operatorname{Im}\eta^{\diff}_{\mathbb{S}^{k}}$.
The first summand is determined in part~(b), while the torsion summands are identified in Theorem~\ref{thm:normal-inv-decomp}. This completes the proof.
\end{proof}

The choice of $\lambda_{m},\mu_{m}$ in Proposition~\ref{prop:kernel-sigma}\,(b)
is fixed for the remainder of the article; in particular, the class
$\widetilde{\eta_{2,k}}$, the constants $T_{m}$ and $R_{2}$
of the equation~\eqref{eq:st-constants}, and the coefficients $s$ and $t$ of 
Theorem~\ref{thm:st-pontryagin} are all understood with respect to it.
\section{Proofs of the main results}\label{sec:proofs}
In this section, we combine the computations of
Sections~\ref{sec:normal} and~\ref{sec:obstr-free} to prove
Theorems~\ref{thm:main-structure}, \ref{thm:st-pontryagin}
and~\ref{thm:image-obstruction}, and Corollary~\ref{cor:infinitely-many}.
\begin{proof}[\textbf{Proof of Theorem~\ref{thm:main-structure}}]
By the surgery exact sequence~\eqref{eq:surgery-even}, the relative smooth
normal invariant map $\eta_{\partial}^{\diff}$ identifies
$\Sdiff_\partial(\OP^2\times\mathbb{D}^k)$ with
$\operatorname{Ker}\sigma^{\diff}_{16+k}$, which is isomorphic to
$\Z^{2}\oplus T_{k}\oplus\Theta_{k}$ by
Proposition~\ref{prop:kernel-sigma}. Under these identifications, each
quadruple $(s,t,t_{1},t_{2})\in\Z^{2}\oplus T_{k}\oplus\Theta_{k}$
corresponds to the unique element of
$\Sdiff_{\partial}(\OP^{2}\times\mathbb{D}^{k})$ whose relative smooth normal
invariant is
$s\,\widetilde{\eta_{1,k}}+t\,\widetilde{\eta_{2,k}}
+\overline{\gamma_{*}}(t_{1})+\eta^{\diff}_{\mathbb{S}^{k}}(t_{2})$.
This completes the proof.
\end{proof}
Before proving Theorem~\ref{thm:st-pontryagin}, we introduce the quantities
$P_{m}$, $Q_{m}$, $T_{m}$, $R_{1}$ and $R_{2}$ occurring in its
statement. Let $m\geq1$ and $k=4m$. Recall from
Theorem~\ref{thm:image-obstruction} the integers $S_{m}$, $A_{m}$, $U_{m}$ and
$g_{m}=\gcd(S_{m},A_{m},U_{m})$, and from
Proposition~\ref{prop:kernel-sigma} the integer $d_{m}=\gcd(A_{m},U_{m})$ together with the integers $\lambda_{m},\mu_{m}$ fixed there, which satisfy
$\lambda_{m}A_{m}+\mu_{m}U_{m}=d_{m}$. Recall also, for $s\geq1$, the integer
$\kappa_{s}$ of~\eqref{eq:kappa-def} and the integer $j_{s}$ of
Theorem~\ref{prop:explicit-gens}. Set
\begin{equation}\label{eq:st-constants}
\begin{aligned}
P_{m}&:=(-1)^{m}\,39\,(2m-1)!\;\kappa_{m}\,j_{m}\,\tfrac{d_{m}}{g_{m}},\\[4pt]
Q_{m}&:=(-1)^{m}\,6\,(2m+3)!\;\kappa_{m}\,j_{m+2}\,\tfrac{U_{m}}{d_{m}},\\[4pt]
T_{m}&:=(-1)^{m}\Bigl(36\,(2m-1)!\;\kappa_{m}\,j_{m}\,\tfrac{d_{m}}{g_{m}}
   -6\,(2m+3)!\;\kappa_{m}\,j_{m+2}\,\tfrac{S_{m}\lambda_{m}}{g_{m}}\Bigr),\\[4pt]
R_{1}&:=648\,j_{2}\,\tfrac{d_{2}}{g_{2}},\\[4pt]
R_{2}&:=450\,j_{2}\,\tfrac{d_{2}}{g_{2}}
   -30240\,j_{4}\,\tfrac{S_{2}\lambda_{2}}{g_{2}}.
\end{aligned}
\end{equation}
All five are integers, since $g_{m}\mid S_{m}$, $g_{m}\mid d_{m}$ and
$d_{m}\mid U_{m}$. As $\kappa_{m}$, $j_{m}$, $d_{m}$ and $g_{m}$ are positive
and $U_{m}=-\tfrac{\kappa_{m}}{8}D_{m+4}\,j_{m+4}\neq0$, the integers $P_{m}$, $Q_{m}$ and
$R_{1}$ are nonzero.
\begin{proof}[\textbf{Proof of Theorem~\ref{thm:st-pontryagin}}]
Set $\eta:=r_{*}\bigl(s\,\widetilde{\eta_{1,k}}+t\,\widetilde{\eta_{2,k}}+x\bigr)
\in[X_{k},BO]$, where $r\colon G/O\to BO$ is the canonical map. 
As $x$ is a torsion element of
$\Ndiff_{\partial}(\OP^{2}\times\mathbb{D}^{k})$ and
$[X_{k},BO]\cong[\Sigma^{k}\OP^{2},BO]\oplus\pi_{k}(BO)\cong\Z^{3}$ is
torsion free by Lemma~\ref{lem:KO-OP2}\,(a)
and~\eqref{eq:KO-pt-groups}, the map $r_{*}$ annihilates $x$. Hence
$\eta=r_{*}\bigl(s\,\widetilde{\eta_{1,k}}+t\,\widetilde{\eta_{2,k}}\bigr)$.
As in Section~\ref{sec:obstr-free}, gluing $\OP^{2}\times\mathbb{D}^{k}$ to $M$ along the boundary diffeomorphism $\partial f$ yields the closed manifold $\widehat{M}=M\cup_{\partial f}(\OP^{2}\times\mathbb{D}^{k})$ together with the degree-one normal map $\widehat{f}=f\cup\operatorname{id}\colon
\widehat{M}\longrightarrow\OP^{2}\times\mathbb{S}^{k}$, covered by a stable bundle map
$\nu_{\widehat{M}}\to\nu_{\OP^{2}\times\mathbb{S}^{k}}\oplus\pi^{*}\eta$,
where $\pi\colon\OP^{2}\times\mathbb{S}^{k}\to X_{k}$ is the collapse map of Section~\ref{sec:obstr-free}. Hence, as stable bundles over
$\OP^{2}\times\mathbb{S}^{k}$,
\begin{equation}\label{eq:st-tangent}
(\widehat{f}^{*})^{-1}T\widehat{M}
=T(\OP^{2}\times\mathbb{S}^{k})-\pi^{*}\eta .
\end{equation}

By Proposition~\ref{prop:kernel-sigma} and the computation of
$\widetilde{\ph}$ in the proof of Theorem~\ref{thm:sigma-on-Sigma-kOP2}, 
\[
\widetilde{\ph}(\eta)=e\,\iota_{k}+a\,\sigma^{k}u+b\,\sigma^{k}u^{2},
\]
where
\begin{align*}
e&=\kappa_{m}\,j_{m}\,\tfrac{d_{m}}{g_{m}}\,t,\\
 a&=\kappa_{m}\,j_{m+2}
   \Bigl(\tfrac{U_{m}}{d_{m}}\,s-\tfrac{S_{m}\lambda_{m}}{g_{m}}\,t\Bigr),\\
 b&=\kappa_{m}\Bigl(\tfrac{j_{m+2}}{240}+c_{m}\Bigr)
   \Bigl(\tfrac{U_{m}}{d_{m}}\,s-\tfrac{S_{m}\lambda_{m}}{g_{m}}\,t\Bigr)
   -\kappa_{m}\,j_{m+4}
   \Bigl(\tfrac{A_{m}}{d_{m}}\,s+\tfrac{S_{m}\mu_{m}}{g_{m}}\,t\Bigr).
\end{align*}
 Since $X_{k}$ is a suspension, the reduced diagonal being null-homotopic, all
cup products of positive-degree classes in $\widetilde{H}^{*}(X_{k};\Q)$
vanish, and~\eqref{phchar} yields
\begin{align}
p_{m}(\eta)&=(-1)^{m+1}(2m-1)!\,e\,\iota_{k},\label{eq:st-pm}\\
p_{m+2}(\eta)&=(-1)^{m+1}(2m+3)!\,a\,\sigma^{k}u,\label{eq:st-pm2}\\
p_{m+4}(\eta)&=(-1)^{m+1}(2m+7)!\,b\,\sigma^{k}u^{2}.\label{eq:st-pm4}
\end{align}

Applying the total Pontryagin class to~\eqref{eq:st-tangent}, and using its
naturality and its multiplicativity, the latter holding since
$H^{*}(\OP^{2}\times\mathbb{S}^{k};\Z)$ is torsion free, yields
\begin{align}
(\widehat{f}^{*})^{-1}p(T\widehat{M})
&=p\bigl(T(\OP^{2}\times\mathbb{S}^{k})\bigr)\cdot\pi^{*}p(\eta)^{-1}
  \notag\\
&=p\bigl(T(\OP^{2}\times\mathbb{S}^{k})\bigr)
  \bigl(1-\pi^{*}p_{m}(\eta)-\pi^{*}p_{m+2}(\eta)-\pi^{*}p_{m+4}(\eta)\bigr)
  \notag\\
&=\bigl(1+6(u\times1)+39(u^{2}\times1)\bigr)
  \bigl(1-\pi^{*}p_{m}(\eta)-\pi^{*}p_{m+2}(\eta)
  -\pi^{*}p_{m+4}(\eta)\bigr).
  \label{eq:st-total-p}
\end{align}
Here $p(\eta)=1+p_{m}(\eta)+p_{m+2}(\eta)+p_{m+4}(\eta)$, since $X_{k}$ has
cells only in the dimensions $k$, $k+8$ and $k+16$;
the second equality follows from the identity $p(\eta)^{-1}=2-p(\eta)$, which
holds since $X_{k}$ is a suspension, so that cup products of positive-degree
classes vanish; and the third uses $p_{2}(\OP^{2})=6u$ and
$p_{4}(\OP^{2})=39u^{2}$ \cite[Theorem~19.4]{borel} together with the stable
triviality of $T\mathbb{S}^{k}$.

Suppose first $k\notin\{8,16\}$, that is $m\notin\{2,4\}$. The classes
$u\times1$, $u^{2}\times1$, $1\times\iota_{k}$ and $u\times\iota_{k}$ then
lie in the distinct degrees $8$, $16$, $k$ and $k+8$, so~\eqref{eq:st-total-p}, \eqref{eq:st-pm} and~\eqref{eq:st-pm2},
together with $\pi^{*}\iota_{k}=1\times\iota_{k}$,
$\pi^{*}(\sigma^{k}u)=u\times\iota_{k}$ 
yield
\begin{align*}
(\widehat{f}^{*})^{-1}p_{m}(T\widehat{M})
&=(-1)^{m}(2m-1)!\,e\,(1\times\iota_{k}),\\
(\widehat{f}^{*})^{-1}p_{m+2}(T\widehat{M})
&=(-1)^{m}\bigl((2m+3)!\,a+6\,(2m-1)!\,e\bigr)(u\times\iota_{k}),\\
(\widehat{f}^{*})^{-1}p_{2}(T\widehat{M})
&=6(u\times1),\\
(\widehat{f}^{*})^{-1}p_{4}(T\widehat{M})
&=39(u^{2}\times1).
\end{align*}
As
$\widehat{f}$ has degree one; therefore, by using
$\langle u^{2}\times\iota_{k},[\OP^{2}\times\mathbb{S}^{k}]\rangle=1$ we obtain
\begin{align*}
(\mathfrak{p}_{4}\,\mathfrak{p}_{m})(\widehat{M})
&=(-1)^{m}\,39\,(2m-1)!\,e=P_{m}\,t,
\end{align*}
\begin{align}\label{jhf}
(\mathfrak{p}_{2}\,\mathfrak{p}_{m+2})(\widehat{M})
&=(-1)^{m}\bigl(6\,(2m+3)!\,a+36\,(2m-1)!\,e\bigr)=Q_{m}\,s+T_{m}\,t .
\end{align}

For $k=16$, that is $m=4$, the class $1\times\iota_{k}$ has degree $16$, so
the degree-$16$ component of~\eqref{eq:st-total-p} is 
\[
(\widehat{f}^{*})^{-1}p_{4}(T\widehat{M})
=39(u^{2}\times1)+7!\,e\,(1\times\iota_{k}).
\]
Squaring, the terms $(u^{2}\times1)^{2}$ and $(1\times\iota_{k})^{2}$ vanish,
since $u^{3}=0$ and $\iota_{k}^{2}=0$, while the two cross terms are equal; hence
\[
\mathfrak{p}_{4}^{2}(\widehat{M})=2\cdot39\cdot7!\,e=2\,P_{4}\,t,\]
while equation~\eqref{jhf} continues to hold.

For $k=8$, that is $m=2$, the classes $u\times1$ and $1\times\iota_{k}$ share
degree $8$, and $u^{2}\times1$ and $u\times\iota_{k}$ share degree $16$, so
the corresponding components of~\eqref{eq:st-total-p} are
\begin{align*}
(\widehat{f}^{*})^{-1}p_{2}(T\widehat{M})
&=6(u\times1)+6\,e\,(1\times\iota_{k}),\\
(\widehat{f}^{*})^{-1}p_{4}(T\widehat{M})
&=39(u^{2}\times1)+\bigl(7!\,a+36\,e\bigr)(u\times\iota_{k});
\end{align*}
using $\kappa_{2}=1$, we obtain
\begin{align*}
\mathfrak{p}_{2}^{3}(\widehat{M})
&=648\,e=R_{1}\,t,\\
(\mathfrak{p}_{2}\,\mathfrak{p}_{4})(\widehat{M})
&=6\bigl(7!\,a+36\,e\bigr)+234\,e=30240\,a+450\,e=Q_{2}\,s+R_{2}\,t .
\end{align*}

As $P_{m}$, $Q_{m}$ and $R_{1}$ are non-zero, in each case the first of the
two relations determines $t$, and the second then determines $s$. This completes the proof.
\end{proof}
\begin{proof}[\textbf{Proof of Corollary~\ref{cor:infinitely-many}}]
Fix $k\equiv 0\pmod 4$ and let $n\in\Z$. By Theorem~\ref{thm:main-structure}
the quadruple $(n,0,0,0)\in\Z^{2}\oplus T_{k}\oplus\Theta_{k}$ corresponds to
the element $[(M_{n},\partial M_{n}),(f_{n},\partial f_{n})]
\in\Sdiff_{\partial}(\OP^{2}\times\mathbb{D}^{k})$
with the relative smooth normal invariant $n\,\widetilde{\eta_{1,k}}$. Set
$W_{n,k}:=M_{n}\cup_{\partial f_{n}}(\OP^{2}\times\mathbb{D}^{k})$. Then
$W_{n,k}$ is a closed smooth manifold of dimension $16+k$, and
$\widehat{f_{n}}=f_{n}\cup\operatorname{id}\colon W_{n,k}\to\OP^{2}\times\mathbb{S}^{k}$ is a
homotopy equivalence.

By Theorem~\ref{thm:st-pontryagin} applied with $(s,t)=(n,0)$, we obtain $(\mathfrak{p}_{2}\,\mathfrak{p}_{m+2})(W_{n,k})=Q_{m}\,n$
in each of its three cases. As $Q_{m}\neq0$, this gives the stated ratio. By the topological invariance of the rational
Pontryagin classes \cite[Theorem~1]{Novikov}, the Pontryagin number
$(\mathfrak{p}_{2}\,\mathfrak{p}_{m+2})(W_{n,k})$ is preserved by
orientation-preserving homeomorphisms and changes sign under
orientation-reversing ones. Consequently, a homeomorphism between $W_{n,k}$
and $W_{n',k}$ forces $Q_{m}n=\pm\,Q_{m}n'$, and $Q_{m}\neq0$ yields
$|n|=|n'|$. Hence the collection $\{W_{n,k}\}_{n\geq1}$ consists of pairwise
non-homeomorphic manifolds.
\end{proof}
\begin{proof}[\textbf{Proof of Theorem~\ref{thm:image-obstruction}}]
As $\sigma^{\diff}_{16+k}$ vanishes on the torsion subgroup $\Ndiff_{\partial}(\OP^{2}\times\mathbb{D}^{k})$,
its image is generated by the values of $\sigma^{\diff}_{16+k}$ on the basis
$\{\widetilde{\xi_{\mathbb{S}^{k}}},\widetilde{\xi_{1,k}},
\widetilde{\xi_{2,k}}\}$ of the free summand. 
By Theorem~\ref{thm:sigma-on-Sigma-kOP2}
$\operatorname{Im}(\sigma^{\diff}_{16+k})=g_{m}\Z$, where
$g_{m}=\gcd(S_{m},A_{m},U_{m})$. Lemma~\ref{lem:2-valuation} determines the $2$-adic valuation $v_{2}(g_{m})$ of $g_m$. Since
$v_2(g_m)\ge 1$, we have $g_m\neq \pm1$, and hence
$\operatorname{Im}(\sigma^{\diff}_{16+k})$ is a proper subgroup of
$\mathbb{Z}$. Therefore $\sigma^{\diff}_{16+k}$ is not surjective.
\end{proof}

\section{Diffeomorphism groups and
\texorpdfstring{$\OP^{2}$}{OP2}-bundles with non-vanishing
\texorpdfstring{$\widehat{\mathfrak{A}}$}{A-hat}-genus}\label{sec:applications}

In this section we determine the rational homotopy
groups of the block diffeomorphism group of~$\OP^2$ in every degree
congruent to $3$ modulo $4$ (Theorem~\ref{pikbdiff}), and of the
diffeomorphism group in degrees $3$, $7$, and $11$
(Corollary~\ref{hgk}).
We then combine this with the surgery obstruction formula of
Theorem~\ref{thm:sigma-general-xi} and an
$\widehat{\mathfrak{A}}$-genus calculation to construct smooth
$\OP^2$-bundles over $\s^4$, $\s^8$, and $\s^{12}$ with
non-vanishing $\widehat{\mathfrak{A}}$-genus
(Theorem~\ref{thm:Ahat-OP2-bundle}).
As a consequence, we deduce that the spaces of metrics of positive
sectional, Ricci, and scalar curvature on~$\OP^2$ have nontrivial
rational homotopy groups in degrees $3$, $7$, and $11$
(Corollary~\ref{cor:positive-curvature}).

Let $M$ be a closed, oriented, smooth manifold. Let
$\operatorname{hAut}(M)$, $\operatorname{Diff}(M)$, and $\widetilde{\operatorname{Diff}}(M)$ denote,
respectively, the topological monoid of orientation-preserving
homotopy equivalences of $M$, the group of orientation-preserving
diffeomorphisms of $M$, and the block diffeomorphism group of $M$
(cf.~\cite{WW01}). Their classifying spaces are denoted by
$\operatorname{BhAut}(M)$, $\operatorname{BDiff}(M)$, and $\operatorname{B\widetilde{Diff}}(M)$,
respectively. 
The classifying spaces $B\widetilde{\operatorname{Diff}}(M)$ and
$\operatorname{BhAut}(M)$ fit into a fibration
\begin{equation}\label{fibration1}
   {\operatorname{hAut}(M)}\big/{\widetilde{\operatorname{Diff}}(M)}
   \;\rightarrow\;
   \operatorname{B\widetilde{Diff}}(M)
   \;\rightarrow\;
   \operatorname{BhAut}(M),
\end{equation}
whose fibre classifies $M$-block bundles that are blockwise homotopy
trivial. 

We now recall a homomorphism defined by evaluating the $\widehat{\mathfrak{A}}$-genus
on the total space of a fibre bundle over a sphere. A homotopy class
$[f] \in \pi_{k}(\operatorname{BDiff}(M))$ corresponds, under the isomorphism
$\pi_{k}(\operatorname{BDiff}(M)) \cong \pi_{k-1}(\operatorname{Diff}(M))$, to a
clutching function
$\phi_{f}\colon (\mathbb{S}^{k-1}, *) \to (\operatorname{Diff}(M), \operatorname{id}_{M})$,
and the associated smooth fibre bundle $M \to E_{f} \to \mathbb{S}^{k}$ has total
space
\[
   E_{f}
   = \bigl(\mathbb{D}^{k}_{+} \times M\bigr)
     \cup_{\Phi_{f}}
     \bigl(\mathbb{D}^{k}_{-} \times M\bigr),
   \qquad
   \Phi_{f}(x, m) = \bigl(x,\, \phi_{f}(x)(m)\bigr),
\]
where $\Phi_{f}\colon \mathbb{S}^{k-1} \times M \to \mathbb{S}^{k-1} \times M$
identifies the boundaries of the two halves; $E_{f}$ is a closed, oriented
$(\dim M + k)$-manifold. As the $\widehat{\mathfrak{A}}$-genus is a bordism invariant,
its evaluation on $E_{f}$ depends only on the homotopy class of  $f$ and yields a homomorphism
\[
   \widehat{\mathfrak{A}}_{\mathrm{Diff}}\colon
   \pi_{k}(\operatorname{BDiff}(M)) \rightarrow \Q,
   \qquad
   [f] \longmapsto \widehat{\mathfrak{A}}(E_{f}).
\]

Since a smooth fibre bundle and its underlying block bundle have the same total
space, $\widehat{\mathfrak{A}}_{\mathrm{Diff}}$ factors as
\begin{equation}\label{eq:Ahat-factor}
   \pi_{k}(\operatorname{BDiff}(M)) \rightarrow
   \pi_{k}(\operatorname{B\widetilde{Diff}}(M))
   \xrightarrow{\ \widehat{\mathfrak{A}}_{\widetilde{\mathrm{Diff}}}\ } \Q ,
\end{equation}
where the first map is induced by the canonical comparison map
$\operatorname{Diff}(M) \to \widetilde{\operatorname{Diff}}(M)$ and
$\widehat{\mathfrak{A}}_{\widetilde{\mathrm{Diff}}}$ is the well-defined homomorphism
given by evaluating the $\widehat{\mathfrak{A}}$-genus on the total space of a block
bundle (see \cite{KKR21}).
\subsection{\texorpdfstring{Rational homotopy of the block diffeomorphism group of $\OP^{2}$}{Rational homotopy of the block diffeomorphism group of OP2}}
In this subsection we compute the rational homotopy groups of the
block diffeomorphism group
$\widetilde{\operatorname{Diff}}(\OP^{2})$ in every degree congruent
to $3$ modulo $4$ (Theorem~\ref{pikbdiff}), and those of the
diffeomorphism group $\operatorname{Diff}(\OP^{2})$ in degrees $3$,
$7$, and $11$ (Corollary~\ref{hgk}); the comparison between block and
smooth diffeomorphism groups used there is needed again in the proof
of Theorem~\ref{thm:Ahat-OP2-bundle}.

\begin{theorem}\label{pikbdiff}
For $k \equiv 0 \pmod{4}$ and $k \geq 0$,
\[
   \pi_{k-1}\!\bigl(\widetilde{\operatorname{Diff}}(\OP^{2})\bigr)
      \otimes \mathbb{Q}
   \;\cong\;
   \pi_{k}\!\bigl(\operatorname{B}\widetilde{\operatorname{Diff}}(\OP^{2})\bigr)
      \otimes \mathbb{Q}
   \;\cong\;
   \begin{cases}
      \mathbb{Q}^{2}, & \text{if } k \notin \{16,\, 24\}, \\[4pt]
      \mathbb{Q}^{3}, & \text{if } k \in \{16,\, 24\}.
   \end{cases}
\]
\end{theorem}
Before we dive into the proof, let us recall that for a manifold $M$ of dimension $\geq 5$, it follows from the $h$-cobordism theorem that
$\pi_{k}\!\bigl(\operatorname{hAut}(M)/\widetilde{\operatorname{Diff}}(M)\bigr)
\cong \Sdiff_{\partial}(M \times \mathbb{D}^{k})$
\cite[Page~33]{stability}.
Moreover, when $M$ is simply connected, It follows from \cite[Proposition~11]{sulliavsresult} that the graded Lie algebra
$\pi_{*}(\operatorname{hAut}(M),\mathrm{id}) \otimes \mathbb{Q}$ is
isomorphic to $H_{*}(\operatorname{Der}_{+}(\Lambda V), \partial)$,
where $\Lambda V$ denotes the Sullivan minimal model~\cite{rational}
of~$M$ and $(\operatorname{Der}_{+}(\Lambda V), \partial)$ is the 
differential graded Lie algebra of positive-degree derivations 
of~$\Lambda V$.
\begin{proof}[Proof of Theorem~\ref{pikbdiff}]
The long exact sequence of homotopy groups associated with the fibration
${\operatorname{hAut}(\OP^{2})}\big/{\widetilde{\operatorname{Diff}}(\OP^{2})} \to \operatorname{B\widetilde{Diff}}(\OP^{2}) \to \operatorname{BhAut}(\OP^{2})$ is given by
\begin{multline}\label{haut}
   \cdots
   \rightarrow
   \pi_{k+1}\!\bigl(\operatorname{hAut}(\OP^{2})\bigr) \otimes \mathbb{Q}
   \rightarrow
   \pi_{k}\!\!\left(
      \frac{\operatorname{hAut}(\OP^{2})}
           {\widetilde{\operatorname{Diff}}(\OP^{2})}
   \right) \otimes \mathbb{Q}
   \rightarrow
   \pi_{k-1}\!\bigl(\widetilde{\operatorname{Diff}}(\OP^{2})\bigr)
      \otimes \mathbb{Q} \\
   \rightarrow
   \pi_{k-1}\!\bigl(\operatorname{hAut}(\OP^{2})\bigr) \otimes \mathbb{Q}
   \rightarrow
   \pi_{k-1}\!\!\left(
      \frac{\operatorname{hAut}(\OP^{2})}
           {\widetilde{\operatorname{Diff}}(\OP^{2})}
   \right) \otimes \mathbb{Q}
   \rightarrow \cdots.
\end{multline}
We note from Theorem~\ref{thm:main-structure} and \cite{stability} that
\begin{equation}\label{haut2}
   \pi_{k}\!\!\left(
      \frac{\operatorname{hAut}(\OP^{2})}
           {\widetilde{\operatorname{Diff}}(\OP^{2})}
   \right) \otimes \mathbb{Q}
   \;\cong\;
   \Sdiff_{\partial}(\OP^{2} \times \mathbb{D}^{k})
      \otimes \mathbb{Q}
   \;=\; \mathbb{Q}^{2},
   \quad \text{for } k \equiv 0 \pmod{4}.
\end{equation}

It remains to compute the rational homotopy groups of
$\operatorname{hAut}(\OP^{2})$. 
The Sullivan minimal model of \(\OP^{2}\) is
\[
(\Lambda(x,y),d),
\qquad |x|=8,\quad |y|=23,\quad dx=0,\quad dy=x^{3}.
\]
Since \(\Lambda(x,y)=\mathbb{Q}[x]\otimes\Lambda(y)\) as a graded vector space, it has basis \(\{x^{j},x^{j}y:j\ge 0\}\). Consequently, $(\Lambda(x,y))^{d}\neq 0$ for $d\in \{8j:j\ge 0\}\cup\{8j+23:j\ge 0\}$.


A derivation \(\theta_n\) of degree \(-n\) is determined by the pair
\((\theta_n(x),\theta_n(y))\), where
\(\theta_n(x)\in(\Lambda V)^{8-n}\) and
\(\theta_n(y)\in(\Lambda V)^{23-n}\).
A degree argument shows that
\(\operatorname{Der}_n(\Lambda V)\neq 0\) for \(n>0\) only when
\(n\in\{7,8,15,23\}\). Moreover, each nonzero derivation space is
one-dimensional, with basis elements given by the following table:
\[
\renewcommand{\arraystretch}{1.3}
\begin{array}{|c|c|c|c|c|}
\hline
   n & 7 & 8 & 15 & 23 \\
   \hline
   \theta_n(x) & 0 & 1 & 0 & 0 \\
   \hline
   \theta_n(y) & x^2 & 0 & x & 1\\
   \hline
\end{array}
\]
The differential on \(\operatorname{Der}(\Lambda V)\) is given by
\[
\partial\theta_n=[d,\theta_n]
=d\circ\theta_n-(-1)^{|\theta_n|}\theta_n\circ d.
\]
Since \(dx=0\) and \(\theta_n(x^3)=3x^2\theta_n(x)\), it follows that
$(\partial\theta_n)(x)=d(\theta_n(x))$,
$(\partial\theta_n)(y)
=d(\theta_n(y))-(-1)^n\,3x^2\theta_n(x)$.
From the description of the derivation spaces, the only possible nonzero differential is $\partial\colon \operatorname{Der}_8(\Lambda V)\rightarrow \operatorname{Der}_7(\Lambda V)$.
For the basis element \(\theta_8\), we obtain $(\partial\theta_8)(x)=0$ and $(\partial\theta_8)(y)
=-3x^2$.
Hence $\partial(\theta_8)=-3\theta_7$.
Therefore
\begin{equation}\label{isoaut}
   \pi_{n}\!\bigl(\operatorname{hAut}(\OP^{2})\bigr) \otimes \mathbb{Q}
   \;\cong\;
   H_n\!\bigl(\operatorname{Der}_+(\Lambda V),\, \partial\bigr)
   \;\cong\;
   \begin{cases}
      \mathbb{Q}, & n = 15, \\[3pt]
      \mathbb{Q}, & n = 23, \\[3pt]
      0,          & \text{otherwise}.
   \end{cases}
\end{equation}

We now deduce the rational homotopy groups $\pi_{k-1}(\widetilde{\operatorname{Diff}}(\OP^{2}))\otimes\mathbb{Q}$
from the long exact sequence~\eqref{haut}, distinguishing two cases according to whether $k-1\in\{15,23\}$ for which $\pi_{k-1}\bigl(\operatorname{hAut}(\OP^{2})\bigl)\otimes \mathbb{Q}$ is non-trivial. If $k\notin\{16,24\}$, then $\pi_{k-1}(\operatorname{hAut}(\OP^2))\otimes\mathbb{Q}=0$ by~\eqref{isoaut}. Hence using ~\eqref{haut2}, the long exact sequence~\eqref{haut} yields the isomorphism  
$\pi_{k-1}(\widetilde{\operatorname{Diff}}(\OP^{2}))\otimes\mathbb{Q}
   \;\cong\;
   \pi_{k}\!\bigl(\operatorname{hAut}(\OP^{2})/\widetilde{\operatorname{Diff}}(\OP^{2})\bigr)\otimes\mathbb{Q}
   \;\cong\;\mathbb{Q}^{2}$.

Now we consider the cases $k=16,24$. Then by the isomorphism~\eqref{isoaut}, we have 
$\pi_{k-1}\bigl(\operatorname{hAut}(\OP^{2})\bigr)\otimes\mathbb{Q}\cong\mathbb{Q}$ and $\pi_k\bigl(\operatorname{hAut}(\OP^{2})\bigr)\otimes\mathbb{Q}=0$. Hence, for $k=16, 24$, the long exact sequence~\eqref{haut} yields  
\begin{equation}\label{segment}
   0
   \rightarrow
\underbrace{\pi_{k}\!\Bigl(
      \tfrac{\operatorname{hAut}(\OP^{2})}{\widetilde{\operatorname{Diff}}(\OP^{2})}
   \Bigr)\otimes\mathbb{Q}}_{\cong\,\mathbb{Q}^{2}}
   \rightarrow
   \pi_{k-1}\bigl(\widetilde{\operatorname{Diff}}(\OP^{2})\bigr)\otimes\mathbb{Q}
   \rightarrow
   \mathbb{Q}
   \rightarrow
   \pi_{k-1}\!\Bigl(
      \tfrac{\operatorname{hAut}(\OP^{2})}{\widetilde{\operatorname{Diff}}(\OP^{2})}
   \Bigr)\otimes\mathbb{Q}.
\end{equation}
Since $\Sdiff_{\partial}(\OP^{2}\times\mathbb{D}^{k-1})
\cong
\pi_{k-1}\!\left(
\tfrac{\operatorname{hAut}(\OP^{2})}
{\widetilde{\operatorname{Diff}}(\OP^{2})}
\right)$,
the desired conclusion follows from the exact sequence~\eqref{segment} once $\Sdiff_{\partial}(\OP^{2}\times\mathbb{D}^{k-1})=0$.

Indeed, we establish stronger result that
\[
\Sdiff_{\partial}(\OP^{2}\times\mathbb{D}^{k-1})\otimes\mathbb{Q}=0 \quad \text{ for every } k\equiv0\pmod4, k\geq 4.
\]
 Since $\OP^{2}\times\mathbb{D}^{k-1}$ is simply
connected of odd dimension $15+k$, $L_{15+k}(\mathbb{Z})=0$, so relative surgery exact sequence for $\OP^{2}\times\mathbb{D}^{k-1}$ becomes
\begin{equation}\label{segment1}
   \Ndiff_{\partial}(\OP^{2}\times\mathbb{D}^{k})
   \xrightarrow{\ \sigma^{\diff}_{16+k}\ }
   L_{16+k}(\mathbb{Z})
   \rightarrow
   \Sdiff_{\partial}(\OP^{2}\times\mathbb{D}^{k-1})
   \rightarrow
   \Ndiff_{\partial}(\OP^{2}\times\mathbb{D}^{k-1})
   \rightarrow 0.
\end{equation}

Since $\mathbb{Q}$ is a flat $\mathbb{Z}$-module, tensoring the exact sequence~\eqref{segment1} with $\mathbb{Q}$ preserves exactness. 
Since
$\Ndiff_{\partial}(\OP^{2}\times\mathbb{D}^{j})
\cong[\Sigma^{j}\OP^{2}\vee\mathbb{S}^{j},\,G/O],
~ j\ge1$,
and
$G/O_{\mathbb{Q}}
\simeq\mathrm{BSO}_{\mathbb{Q}}
\simeq\prod_{i\ge1}K(\mathbb{Q},4i)$
\cite[p.~433, 807]{LM24}, we obtain $[X,G/O]\otimes\mathbb{Q}
\cong
\bigoplus_{i\ge1}\widetilde H^{4i}(X;\mathbb{Q})$
for every finite complex $X$. Therefore,
\[
\Ndiff_{\partial}(\OP^{2}\times\mathbb{D}^{k-1})\otimes\mathbb{Q}=0,
\qquad
\Ndiff_{\partial}(\OP^{2}\times\mathbb{D}^{k})\otimes\mathbb{Q}
\cong\mathbb{Q}^{3},
\]
since $\widetilde{H}^{*}(\Sigma^{j}\OP^{2}\vee\mathbb{S}^{j};\Q)$ is nonzero
only in degrees $j$, $j+8$, and $j+16$.
Since $\sigma^{\diff}_{16+k}$ has nonzero image by Theorem~\ref{thm:image-obstruction}, the map $\sigma^{\diff}_{16+k}\otimes\mathbb{Q}\colon \mathbb{Q}^3\to\mathbb{Q}$ is surjective. Thus after rationalization the exact sequence~\eqref{segment1} yields $\Sdiff_{\partial}(\OP^{2}\times\mathbb{D}^{k-1})\otimes\mathbb{Q}=0$.
This completes the proof.
\end{proof}



\begin{corollary}\label{hgk}
 $ \pi_{k-1}\!\bigl(\operatorname{Diff}(\OP^{2})\bigr) \otimes \mathbb{Q}
   \;\cong\;
   \pi_{k}\!\bigl(\operatorname{BDiff}(\OP^{2})\bigr) \otimes \mathbb{Q}
   \;\cong\;
   \mathbb{Q}^{2}$, for $k = 4$, $8$, and $12$.
\end{corollary}
\begin{proof}
Let $\widetilde{\operatorname{Diff}}(\OP^{2})/\operatorname{Diff}(\OP^{2})$ denote the homotopy fibre of the map $\operatorname{BDiff}(\OP^{2}) \to \operatorname{B\widetilde{Diff}}(\OP^{2})$ induced by the canonical comparison map $\operatorname{Diff}(\OP^{2}) \rightarrow \widetilde{\operatorname{Diff}}(\OP^{2})$. Since $\OP^{2}$ is $7$-connected, it follows from \cite[Remark~2.4]{georg} that the $\pi_{l}\bigl(\widetilde{\operatorname{Diff}}(\OP^{2})/\operatorname{Diff}(\OP^{2})\bigr) \otimes \mathbb{Q} = 0$ for all $l \leq 12$. 

Consequently, the long exact sequence in rational homotopy associated with the fibration
$ \widetilde{\operatorname{Diff}}(\OP^{2}) \big/ \operatorname{Diff}(\OP^{2})
   \;\rightarrow\;
   \operatorname{BDiff}(\OP^{2})
   \;\rightarrow\;
   \operatorname{B\widetilde{Diff}}(\OP^{2})$
implies that the induced homomorphism $\pi_{l}(\operatorname{BDiff}(\OP^{2})) \otimes \mathbb{Q} \to \pi_{l}(\operatorname{B\widetilde{Diff}}(\OP^{2})) \otimes \mathbb{Q}$ is an isomorphism for all $l \leq 12$. The desired isomorphisms then follow directly from Theorem~\ref{pikbdiff}. 
\end{proof}

\subsection{\texorpdfstring{$\OP^2$-bundles with non-vanishing $\widehat{\mathfrak{A}}$-genus}{OP2-bundles with non-vanishing A-hat-genus}}
We now prove Theorem~\ref{thm:Ahat-OP2-bundle}. The argument rests on
the linear independence of two forms in the Pontryagin character
coordinates $a_{\xi}$, $b_{\xi}$ of~\eqref{eq:ph-coords}, namely the
surgery obstruction of Theorem~\ref{thm:sigma-general-xi} and the
$\widehat{\mathfrak{A}}$-genus of the associated closed manifold,
which allows a stable bundle over $\Sigma^{k}\OP^{2}$ to be chosen
with vanishing surgery obstruction but non-vanishing
$\widehat{\mathfrak{A}}$-genus.
\begin{proof}[\textbf{Proof of Theorem~\ref{thm:Ahat-OP2-bundle}}]
Fix $k \in \{4, 8, 12\}$ and write $k = 4m$. Since $\operatorname{Ker}[\Sigma^{k}\OP^{2},J]$ has finite index in
$[\Sigma^{k}\OP^{2},BO]\cong\Z^{2}$ (see the proof of
Lemma~\ref{lem:normal-splitting}) and $\widetilde{\ph}\otimes\Q:
\widetilde{KO}^{0}(\Sigma^{k}\OP^{2})\otimes\Q
\xrightarrow{\cong}
\Q\{\sigma^{k}u\}\oplus\Q\{\sigma^{k}u^{2}\}$ is an isomorphism \cite[p.~7]{kotheory}, for every $(a,b)\in\Q^{2}$ there exists
$\xi\in\operatorname{Ker}[\Sigma^{k}\OP^{2},J]$ and a nonzero integer $\lambda$ satisfying
\begin{equation}\label{pointry2}
\widetilde{\ph}(\xi)
=
a_{\xi}\,\sigma^{k}u+b_{\xi}\,\sigma^{k}u^{2},
\end{equation}
where $(a_{\xi},b_{\xi})\,=\lambda (a,b)$.
Hence, by the exact sequence~\eqref{eq:N-SES}, $\xi$ admits a lift
$[\widetilde{\xi}]\in[\Sigma^{k}\OP^{2},G/O]$. It now follows from Theorem~\ref{thm:sigma-general-xi},
the surgery obstruction of the lift $[\widetilde{\xi}]$ of $\xi$ to
$[\Sigma^{k}\OP^{2}, G/O]$ is given by
\begin{equation}\label{eq:sigma-ab}
   \sigma^{\diff}_{16+k}([\widetilde{\xi}])
   \;=\; -\,\frac{1}{8}\!\left(
      \frac{14\,D_{m+2}}{15}\,a_{\xi} \;+\; D_{m+4}\,b_{\xi}
   \right).
\end{equation}
If the surgery obstruction vanishes, then by relative surgery exact sequence $[\widetilde{\xi}]$ lifts to
$\Sdiff_{\partial}(\OP^{2}\times\mathbb{D}^{k})$, and hence there exists a
degree-one normal map $(f,\partial f)\colon\,(M,\partial M)\rightarrow (\OP^{2}\times\mathbb{D}^{k},\OP^{2}\times\mathbb{S}^{k-1})$ as in
diagram~\eqref{normalsquare}, with $f$ a homotopy equivalence and
$\partial f$ a diffeomorphism.

Under the isomorphism $\Sdiff_{\partial}(\OP^{2}\times\mathbb{D}^{k})
\cong
\pi_{k}\bigl(\operatorname{hAut}(\OP^{2})/
\widetilde{\operatorname{Diff}}(\OP^{2})\bigr)$ \cite[p.~33]{stability}, the class $[(M,\partial M),(f,\partial f)]$ determines an element
$x\in\pi_{k}(\operatorname{B}\widetilde{\operatorname{Diff}}(\OP^{2}))$ via the map induced by the fibration~\eqref{fibration1}. Writing
$\mathbb{S}^{k}=\mathbb{D}^{k}_{+}\cup_{\mathbb{S}^{k-1}}\mathbb{D}^{k}_{-}$,
the block bundle classified by $x$ is obtained by gluing the block $M$ over
$\mathbb{D}^{k}_{+}$ to the trivial block $\OP^{2}\times\mathbb{D}^{k}_{-}$
along the boundary diffeomorphism $\partial f$; its total space is therefore
the closed manifold $\widehat{M}=M\cup_{\partial f}(\OP^{2}\times\mathbb{D}^{k}_{-})$,
and its projection $\pi_{E}\colon\widehat{M}\to\mathbb{S}^{k}$ restricts to the
canonical block projections over $\mathbb{D}^{k}_{\pm}$. That $x$ maps to the
trivial element of $\pi_{k}(\operatorname{BhAut}(\OP^{2}))$ is witnessed by the
fibrewise homotopy equivalence $\widehat{f}=f\cup\operatorname{id}\colon
\widehat{M}\xrightarrow{\ \simeq\ }\OP^{2}\times\mathbb{S}^{k}$, which satisfies $\operatorname{proj}_{2}\circ\widehat{f}=\pi_{E}$. Since $\widehat{\mathfrak{A}}_{\widetilde{\mathrm{Diff}}}(x)$ is given by
evaluating the $\widehat{\mathfrak{A}}$-genus on the total space of the block
bundle classified by $x$, therefore,
$\widehat{\mathfrak A}_{\widetilde{\mathrm{Diff}}}(x)
=
\widehat{\mathfrak A}(\widehat M)$.

We now compute the $\widehat{\mathfrak{A}}$-genus of the closed manifold
$\widehat{M}=M\cup_{\partial M}(\OP^{2}\times\mathbb{D}^{k})$. Since the degree one normal map $\widehat{f}\colon \widehat{M}\rightarrow \OP^{2}\times\mathbb{S}^{k}$ is covered by a bundle map
$\nu_{\widehat{M}}\to\nu_{\OP^{2}\times\mathbb{S}^{k}}\oplus\pi^{*}\xi$.
Arguing exactly as in the computation of the $\mathcal{L}$-class in
Section~\ref{sec:obstr-free}, the naturality and multiplicativity of the
$\widehat{\mathfrak{A}}$-class, together with
$\widehat{\mathfrak{A}}(T\mathbb{S}^{k})=1$, yield
\begin{equation}\label{eq:Ahat-eval}
\begin{aligned}
\widehat{\mathfrak{A}}(\widehat{M})
&=
\Bigl\langle
\widehat{f}^{*}\!\left(
\mathrm{proj}_{1}^{*}\widehat{\mathfrak{A}}(\OP^{2})
\cdot
\pi^{*}\bigl(\widehat{\mathfrak{A}}(\xi)^{-1}\bigr)
\right),
[\widehat{M}]
\Bigr\rangle \\
&=
\Bigl\langle
\mathrm{proj}_{1}^{*}\widehat{\mathfrak{A}}(\OP^{2})
\cdot
\pi^{*}\bigl(\widehat{\mathfrak{A}}(\xi)^{-1}\bigr),
[\OP^{2}\times\mathbb{S}^{k}]
\Bigr\rangle \\
&=
-
\Bigl\langle
\mathrm{proj}_{1}^{*}\widehat{\mathfrak{A}}(\OP^{2})
\cdot
\pi^{*}\bigl(\widehat{\mathfrak{A}}(\xi)-2\bigr),
[\OP^{2}\times\mathbb{S}^{k}]
\Bigr\rangle,
\end{aligned}
\end{equation}
where the second equality uses that $\widehat{f}$ has degree one, and the last equality
follows from the identity
$\widehat{\mathfrak{A}}(\xi)^{-1}=2-\widehat{\mathfrak{A}}(\xi)$, which holds
because every cup product of positive-degree classes in
$\widetilde{H}^{*}(X_{k};\Q)$ vanishes.

Recall that the $\widehat{\mathfrak{A}}$-genus is the multiplicative
sequence $\{\widehat{\mathfrak{A}}_{s}\}$ associated with the
characteristic power series $Q(z)=\frac{\sqrt{z}/2}{\sinh(\sqrt{z}/2)}$
\cite[p.~231]{spingeo}. The coefficient of the top Pontryagin class
$p_{s}$ in $\widehat{\mathfrak{A}}_{s}$ is $-\frac{B_{s}}{2(2s)!}$
(this coefficient is $(-1)^{s-1}s$ times the coefficient of $z^{s}$ in
$\log\,Q(z)$ \cite[p.~11--14]{Hir66},
which equals $\frac{(-1)^{s}B_{s}}{2s\,(2s)!}$ by the expansion
$\log(t/\sinh t)
=\sum_{s\geq1}\tfrac{(-1)^{s}2^{2s}B_{s}}{2s\,(2s)!}\,t^{2s}$
evaluated at $t=\sqrt{z}/2$). Arguing as in
Lemma~\ref{lem:L-eq-D-ph}, for any stable bundle $\eta$ over
$\Sigma^{k}\OP^{2}$ we have
\begin{equation}\label{eq:Ahat-eq-ahat-ph}
   \widehat{\mathfrak{A}}_{s}(\eta) \;=\; \widehat{a}_{s}\cdot\ph_{2s}(\eta),
   \qquad \text{where} \quad \widehat{a}_{s} = \frac{(-1)^{s}\,B_{s}}{4s}.
\end{equation}

Since \(p_{2}(\OP^{2})=6u\) and \(p_{4}(\OP^{2})=39u^{2}\), where \(u\) is the generator of \(H^{8}(\OP^{2};\Z)\) \cite[Theorem~19.4]{borel}, it follows that $\widehat{\mathfrak{A}}(\OP^{2})=1-\frac{1}{240}u$. By an analogous argument as in the proof of Theorem~\ref{thm:sigma-general-xi}, with $D_s$ replaced by $\hat a_s$ and $\mathcal{L}(\OP^{2})$ by $\widehat{\mathfrak{A}}(\OP^{2})$, together with
\eqref{eq:Ahat-eq-ahat-ph} and~\eqref{eq:Ahat-eval}, we obtain

\begin{equation}\label{eq:Ahat-ab}
   \widehat{\mathfrak{A}}(\widehat{M})
   \;=\; \frac{\widehat{a}_{m+2}}{240}\,a_{\xi} \;-\; \widehat{a}_{m+4}\,b_{\xi}.
\end{equation}
The surgery obstruction~\eqref{eq:sigma-ab} and the
$\widehat{\mathfrak{A}}$-genus~\eqref{eq:Ahat-ab} are linear forms in
$(a_{\xi}, b_{\xi})$, and are linearly independent precisely when the determinant
\begin{equation*}
\Delta_k
:=
\det
\begin{pmatrix}
\dfrac{14D_{m+2}}{15} & D_{m+4}\\[8pt]
\dfrac{\hat a_{m+2}}{240} & -\hat a_{m+4}
\end{pmatrix}
\end{equation*}
is nonzero. Substituting the expressions for $D_s$ and $\hat a_s$ from
Lemma~\ref{lem:L-eq-D-ph} and~\eqref{eq:Ahat-eq-ahat-ph}, respectively, and
simplifying, we obtain
\begin{equation*}
\Delta_k
=
\frac{2^{2m+1}(2^{2m+4}-1)B_{m+2}B_{m+4}}
{(m+2)(m+4)}
>0 \quad \text{for every } m \geq 1,
\end{equation*}
since every $B_s$ is positive. 
This implies that there exists a pair $(a_{\xi},b_{\xi})\in \Q^2$ and an element $\xi\in\operatorname{Ker}[\Sigma^{k}\OP^{2},\,J]$ satisfying equation~\eqref{pointry2}, such that $\sigma^{\diff}_{16+k}([\widetilde{\xi}]) = 0$, while $\widehat{\mathfrak{A}}(\widehat{M})
   \neq 0$. Consequently $\widehat{\mathfrak{A}}_{\widetilde{\mathrm{Diff}}}(x) \neq 0$,
so $x$ has non-zero image in
$\pi_{k}\bigl(B\widetilde{\operatorname{Diff}}(\OP^{2})\bigr) \otimes \Q$.

Finally, following \cite[Remark~2.4]{georg}, the canonical comparison
map
$\operatorname{Diff}(\OP^{2}) \to \widetilde{\operatorname{Diff}}(\OP^{2})$
induces an isomorphism
$\pi_{k}\bigl(B\!\operatorname{Diff}(\OP^{2})\bigr) \otimes \Q
\xrightarrow{\ \cong\ }
\pi_{k}\bigl(B\widetilde{\operatorname{Diff}}(\OP^{2})\bigr) \otimes \Q$
for $k = 4, 8, 12$. Therefore combining this with the factorization~\eqref{eq:Ahat-factor}, we conclude that there exists a smooth oriented $\OP^{2}$-bundle $E\rightarrow\mathbb{S}^{k}$ and a nonzero integer $N$ such that $\widehat{\mathfrak{A}}(E)=N\cdot\widehat{\mathfrak{A}}_{\rm\widetilde{Diff}}(x)$, which is nonzero. This completes the proof.
\end{proof}
\begin{remark}
For every $k\equiv0\pmod4$, there exists an $\OP^{2}$-block bundle over
$\mathbb{S}^{k}$ with non-zero $\widehat{\mathfrak A}$-genus. In
particular,
$\pi_k(B\widetilde{\operatorname{Diff}}(\OP^{2}))$
contains a subgroup isomorphic to~$\Z$.

\end{remark}
For a closed smooth manifold $M$, let 
$\mathcal{R}^{\mathrm{sec}>0}(M) \subset \mathcal{R}^{\mathrm{Ric}>0}(M) \subset \mathcal{R}^{\mathrm{scal}>0}(M)$ 
denote the subspaces of the space of all Riemannian metrics on $M$ with positive sectional, Ricci, and scalar curvature, respectively.

Moreover, if $M$ is a spin manifold, we fix a spin structure
$\mathfrak{s}$ on $M$. Let
$\widetilde{\mathrm{Diff}}^{\mathrm{spin}}(M)
=
\{\,f\in\mathrm{Diff}(M)\mid f^{*}\mathfrak{s}\cong\mathfrak{s}\,\}$
denote the subgroup of diffeomorphisms preserving $\mathfrak{s}$ up to
isomorphism. The \emph{spin diffeomorphism group}
$\mathrm{Diff}^{\mathrm{spin}}(M)$ consists of pairs $(f,\widehat{f})$,
where $f\in\widetilde{\mathrm{Diff}}^{\mathrm{spin}}(M)$ and
$\widehat{f}\colon f^{*}\mathfrak{s}\to\mathfrak{s}$ is an isomorphism of
spin structures. The forgetful map
\[
p\colon
\mathrm{Diff}^{\mathrm{spin}}(M)
\rightarrow
\widetilde{\mathrm{Diff}}^{\mathrm{spin}}(M),
\qquad
(f,\widehat{f})\longmapsto f,
\]
is a double covering (see \cite[p.~696]{spanior}). Consequently, the
induced homomorphism
\[p_{*}\colon
\pi_{i}\bigl(\mathrm{Diff}^{\mathrm{spin}}(M)\bigr)
\rightarrow
\pi_{i}\bigl(\widetilde{\mathrm{Diff}}^{\mathrm{spin}}(M)\bigr)\] is an isomorphism for all $i\ge 2$.

Let $d=\dim M$, let $g_{0}\in\mathcal{R}^{\mathrm{scal}>0}(M)$, and let $l$ be an integer such that $m:=4l-d-1\geq1$. By \cite[p.~3997--3999]{botvik}, the composition
\begin{equation}\label{compose2}
\pi_{m}\bigl(\mathrm{Diff}^{\mathrm{spin}}(M)\bigr)
\xrightarrow{\;(\mathrm{ev}_{g_{0}})_{*}\;}
\pi_{m}\bigl(\mathcal{R}^{\mathrm{scal}>0}(M),g_{0}\bigr)
\xrightarrow{\;(\mathrm{inddiff}_{g_{0}})_{*}\;}
KO^{-4l}(\mathrm{pt})\cong\mathbb{Z}
\end{equation}
sends a homotopy class $[f]$ to $a_{l}\,\widehat{\mathfrak{A}}_{\rm Diff}(f)$, where $a_{l}=1$ if $l$ is even and $a_{l}=\tfrac12$ if $l$ is odd. Here
\[
\mathrm{ev}_{g_{0}}\colon \mathrm{Diff}^{\mathrm{spin}}(M)\rightarrow
\mathcal{R}^{\mathrm{scal}>0}(M),\qquad
(f,\widehat{f})\longmapsto f^{*}g_{0},
\]
is the evaluation map, and 
\[
\mathrm{inddiff}_{g_{0}}\colon
\mathcal{R}^{\mathrm{scal}>0}(M)\rightarrow
\Omega^{\infty+d+1}KO
\]
is the secondary index map associated to the base metric $g_{0}$. 


If $g_{0}$ has positive Ricci curvature, then the evaluation map
$
\mathrm{ev}_{g_{0}}\colon
\mathrm{Diff}^{\mathrm{spin}}(M)\rightarrow
\mathcal{R}^{\mathrm{scal}>0}(M)
$
takes values in $\mathcal{R}^{\mathrm{Ric}>0}(M)$. Hence the above composition~\eqref{compose2} factors through $\pi_{m}(\mathcal{R}^{\mathrm{Ric}>0}(M),\, g_{0})$, and therefore the existence of
$[f]\in\pi_{m}\bigl(\mathrm{Diff}^{\mathrm{spin}}(M)\bigr)$ with
$\widehat{\mathfrak{A}}_{\rm Diff}(f)\neq0$ implies that
\[
\pi_{m}\bigl(\mathcal{R}^{\mathrm{Ric}>0}(M)\bigr)\otimes\mathbb{Q}\neq0
\]
The same argument applies to any $\mathrm{Diff}(M)$-invariant subspace of $\mathcal{R}^{\mathrm{scal}>0}(M)$ containing $g_{0}$. In particular, if $g_{0}$ has positive sectional curvature, then
\[
\pi_{m}\bigl(\mathcal{R}^{\mathrm{sec}>0}(M)\bigr)\otimes\mathbb{Q}\neq0
\]

Note that $\OP^{2}$ is compact rank-one symmetric space, and hence its canonical homogeneous metric $g_{\mathrm{st}}$ has positive sectional curvature \cite[p.~208--210]{peterson}. In particular,
\[
\mathcal{R}^{\mathrm{sec}>0}(\OP^{2}),\quad
\mathcal{R}^{\mathrm{Ric}>0}(\OP^{2}),\quad\text{and}\quad
\mathcal{R}^{\mathrm{scal}>0}(\OP^{2})
\]
are all non-empty. Moreover, since $\OP^{2}$ is $7$-connected, its first and second Stiefel--Whitney classes vanish, implying that it admits a unique spin structure, which we denote by $\mathfrak{s}$. Consequently,
$f^{*}\mathfrak{s}\cong\mathfrak{s}$ for every
$f\in\mathrm{Diff}(\OP^{2})$, and therefore
$\widetilde{\mathrm{Diff}}^{\mathrm{spin}}(\OP^{2})
=
\mathrm{Diff}(\OP^{2}).$


\begin{corollary}\label{cor:positive-curvature}
For each $k\in\{3,7,11\}$, each of the homotopy groups
\[
\pi_{k}\bigl(\mathcal{R}^{\mathrm{sec}>0}(\OP^{2}),g_{\mathrm{st}}\bigr),\quad
\pi_{k}\bigl(\mathcal{R}^{\mathrm{Ric}>0}(\OP^{2}),g_{\mathrm{st}}\bigr),\quad
\text{and}\quad
\pi_{k}\bigl(\mathcal{R}^{\mathrm{scal}>0}(\OP^{2}),g_{\mathrm{st}}\bigr)
\]
contains an element of infinite order.
\end{corollary}
\begin{proof}
Fix $k\in\{3,7,11\}$ and let $l$ be defined by $4l=k+17$. Then $k=4l-\dim\OP^{2}-1$. By Theorem~\ref{thm:Ahat-OP2-bundle}, there exists an element $x\in\pi_k(\mathrm{Diff}(\OP^{2}))$ with $\widehat{\mathfrak{A}}_{\rm Diff}(x)\neq0$.
Since $\widetilde{\mathrm{Diff}}^{\mathrm{spin}}(\OP^{2})=\mathrm{Diff}(\OP^{2})$ and
$k\ge2$, the forgetful map $p$ induces an isomorphism
\[
p_{*}\colon
\pi_k(\mathrm{Diff}^{\mathrm{spin}}(\OP^{2}))
\rightarrow
\pi_k(\mathrm{Diff}(\OP^{2})).
\]
Identifying $x$ with its preimage under $p_{*}$, the preceding discussion shows that
the image of $x$ in each of
\[
\pi_k\bigl(\mathcal{R}^{\mathrm{sec}>0}(\OP^{2}),g_{\mathrm{st}}\bigr),\quad
\pi_k\bigl(\mathcal{R}^{\mathrm{Ric}>0}(\OP^{2}),g_{\mathrm{st}}\bigr),
\quad\text{and}\quad
\pi_k\bigl(\mathcal{R}^{\mathrm{scal}>0}(\OP^{2}),g_{\mathrm{st}}\bigr)
\]
is mapped to $a_l\,\widehat{\mathfrak{A}}_{\rm Diff}(x)\neq0$ under
$(\mathrm{inddiff}_{g_{\mathrm{st}}})_{*}$. Hence each of these homotopy groups
contains an element of infinite order.
\end{proof}
\appendix
\section{Arithmetic lemmas}\label{sec:appendix1}
In this appendix we prove the two arithmetic results about Bernoulli
numbers used in the article; Lemma~\ref{integrality}, which shows
that the right-hand side of the congruence~\eqref{eq:cm-congruence} is an integer,
and Lemma~\ref{lem:2-valuation}, which computes the $2$-adic
valuation stated in Theorem~\ref{thm:image-obstruction}. Throughout,
$v_{p}$ denotes the $p$-adic valuation, extended to $\Q^{\times}$
with the convention $v_{p}(0)=\infty$.


Kummer's congruence is formulated in \cite{classical} for the
coefficients of $x/(e^{x}-1)$. Accordingly,
define rational numbers $\widetilde{B}_{t}$ (denoted by $\beta_{t}$ in
\cite[p.~90]{hardy}) by the expansion
\[
\frac{x}{e^{x}-1}=\sum_{t=0}^{\infty}\widetilde{B}_{t}\,\frac{x^{t}}{t!}.
\]
Then $\widetilde{B}_{0}=1$, $\widetilde{B}_{1}=-\tfrac{1}{2}$,
$\widetilde{B}_{2t+1}=0$ for $t\geq 1$, and
\begin{equation}\label{bernoulli}
    \widetilde{B}_{2s}=(-1)^{s-1}B_{s} \qquad\text{for } s\geq 1.
\end{equation}
\begin{lemma}\label{integrality}
Let $n_s$, $j_s$ be as in Theorem~\ref{prop:explicit-gens}.
Then for every integer $m\geq 1$,
\[
240 \;\Big|\; \bigl(n_{m+4}\,j_{m+2}-n_{m+2}\,j_{m+4}\bigr).
\]
\end{lemma}
\begin{proof}
For $s\geq 1$, set $E_{s}:=B_{s}/(4s)=n_{s}/j_{s}$, where by~\eqref{bernoulli}, $E_{s}=(-1)^{s-1}\widetilde{B}_{2s}/(4s)$.
Put $s_{1}=m+2$, $s_{2}=m+4$ and
$\Lambda:=n_{s_{2}}j_{s_{1}}-n_{s_{1}}j_{s_{2}}$. Since
$240=2^{4}\cdot3\cdot5$, it suffices to show $v_{2}(\Lambda)\geq 4$ and
$v_{p}(\Lambda)\geq 1$ for $p=3,5$. Recall from
\cite[Theorem~2.6]{Ada65a} that
\begin{equation}\label{eq:val-j}
v_{p}(j_{s})=
\begin{cases}
1+v_{p}(s) & \text{if }(p-1)\mid 2s,\\
0 & \text{otherwise},
\end{cases}
\quad\text{for $p$ odd},
\qquad
v_{2}(j_{s})=3+v_{2}(s),
\end{equation}
In particular, $j_{s}$ is even for every $s\geq 1$, and hence each
$n_{s}$ is odd, since $\gcd(n_{s},j_{s})=1$. Note also that clearing denominators gives
\begin{equation}\label{eq:Lambda-identity}
\Lambda=j_{s_{1}}j_{s_{2}}\bigl(E_{s_{2}}-E_{s_{1}}\bigr).
\end{equation}

If $m$ is even, then $s_{1},s_{2}$ are even, so~\eqref{eq:val-j} gives $v_{2}(j_{s_{i}})\geq 4$,
$v_{3}(j_{s_{i}})\geq 1$ and $v_{5}(j_{s_{i}})\geq 1$; hence
$240$ divides both terms of $\Lambda$.

Suppose $m$ is odd. Then $s_{1},s_{2}$ are odd. For $p=3$,~\eqref{eq:val-j} again gives $v_{3}(j_{s_{i}})\geq 1$, whence
$v_{3}(\Lambda)\geq 1$. For $p=2$, we have $v_{2}(j_{s_{i}})=3$, and
each $n_{s_{i}}$ is odd since $\gcd(n_{s_{i}},j_{s_{i}})=1$. Writing $j_{s_{i}}=8u_{i}$ with $u_{i}$ odd, we get
$\Lambda=8(n_{s_{2}}u_{1}-n_{s_{1}}u_{2})$, and
$v_{2}(n_{s_{2}}u_{1}-n_{s_{1}}u_{2})\geq 1$ since $n_{s_{i}}u_{j}$ are
all odd; hence $v_{2}(\Lambda)\geq 4$.

For $p=5$,~\eqref{eq:val-j} gives $v_{5}(j_{s_{i}})=0$, so by~\eqref{eq:Lambda-identity},
$v_{5}(\Lambda)=v_{5}(E_{s_{2}}-E_{s_{1}})$, where now
$E_{s_{i}}=\widetilde{B}_{2s_{i}}/(4s_{i})$. Since $m$ is odd,
$2s_{1}\equiv 2s_{2}\equiv 2\pmod 4$; in particular $4\nmid 2s_{i}$,
so the $\widetilde{B}_{2s_{i}}/(2s_{i})$ are $5$-integer
\cite[Proposition~15.2.4]{classical} and Kummer's congruence
\cite[Chapter~15, Theorem~5]{classical} yields
\[
\frac{\widetilde{B}_{2s_{2}}}{2s_{2}}\equiv\frac{\widetilde{B}_{2s_{1}}}{2s_{1}}
\pmod{5}.
\]
As $2$ is a unit modulo $5$, it follows that
$E_{s_{2}}\equiv E_{s_{1}}\pmod 5$, whence $v_{5}(\Lambda)\geq 1$.
This completes the proof.
\end{proof}
\begin{lemma}\label{lem:2-valuation}
Let $k=4m$ with $m\geq 1$, and let $S_{m},A_{m},U_{m}$ be as in
Theorem~\ref{thm:image-obstruction}. Set
\[
  e_m:=2m-2+v_2(\kappa_m)=
  \begin{cases}
    2m-2, & m\ \mathrm{even},\\
    2m-1, & m\ \mathrm{odd},
  \end{cases}
\]
where $\kappa_m$ is as in~\eqref{eq:kappa-def}. Then
\[
  v_2(S_{m})=e_m,\qquad v_2(A_{m})=e_m+4,\qquad v_2(U_{m})=e_m+8.
\]
Consequently, $v_2\bigl(\gcd(S,A,B)\bigr)=e_m$.
\end{lemma}

\begin{proof}
By~\eqref{eq:val-j},
$v_2(j_s)=3+v_2(s)$ and each $n_s$ is odd. Substituting $B_s=4s\,n_s/j_s$ into the
formula for $D_s$ in Lemma~\ref{lem:L-eq-D-ph} gives
\[
   D_s=(-1)^{s+1}\,2^{2s+1}\bigl(2^{2s-1}-1\bigr)\frac{n_s}{j_s},
\]
whence $v_2(D_s)=2s-2-v_2(s)$ and $v_2(D_s\,j_s)=2s+1$. \eqref{eq:sigma-xi-S} and~\eqref{eq:sigma-xi2} then give
\[
   v_2(S_{m})=v_2(\kappa_m)-3+(2m+1)=e_m,
   \qquad
   v_2(U_{m})=v_2(\kappa_m)-3+(2m+9)=e_m+8.
\]

By~\eqref{eq:sigma-xi1}, $A_{m}=-\tfrac{\kappa_m}{8}(P+Q+R)$, where
\[
   P=\tfrac{14}{15}\,D_{m+2}\,j_{m+2},\qquad
   Q=\tfrac{1}{240}\,D_{m+4}\,j_{m+2},\qquad
   R=D_{m+4}\,c_m.
\]
Using $v_2(D_{m+2}\,j_{m+2})=2m+5$ and $v_2\bigl(\tfrac{14}{15}\bigr)=1$,  we obtain $v_2(P)=2m+6$. Setting $\iota:=\tfrac{j_{m+2}}{240}+c_m$, we have
$Q+R=D_{m+4}\,\iota$. By~\eqref{eq:cm-congruence},
\[
   y:=\frac{n_{m+4}}{j_{m+4}}\,\iota-\frac{n_{m+2}}{240}
\]
is an integer. Consequently,
$\iota=\frac{j_{m+4}}{n_{m+4}}\bigl(y+\frac{n_{m+2}}{240}\bigr)$. 
Since $y$ is an integer, $v_2(y)\geq 0$, whereas
$v_2\bigl(\tfrac{n_{m+2}}{240}\bigr)=-4$. Hence $v_2\bigl(y+\tfrac{n_{m+2}}{240}\bigr)=-4$. As $n_{m+4}$ is odd, this yields
$v_2(\iota)=v_2(j_{m+4})-4=v_2(m+4)-1$. Combining this with
$v_2(D_{m+4})=2m+6-v_2(m+4)$ we obtain $v_2(Q+R)=2m+5$. Since
$v_2(P)=2m+6>2m+5=v_2(Q+R)$, it follows that $v_2(P+Q+R)=2m+5$, and hence
\[
   v_2(A_{m})=v_2(\kappa_m)-3+(2m+5)=e_m+4.
\]
Consequently
$v_2\bigl(\gcd(S_{m},A_{m},U_{m})\bigr)=\min\{e_m,\,e_m+4,\,e_m+8\}=e_m$. This completes the proof.
\end{proof}

We record only the $2$-primary valuation, as it is the only one determined in closed
form. For an odd prime $p$ one has $v_p(g_{m})=\min\{v_p(S_{m}),v_p(_{m}A),v_p(U_{m})\}$, and since
$\operatorname{odd}(S_{m})=(2^{2m-1}-1)\,n_m$ and $\operatorname{odd}(U_{m})=(2^{2m+7}-1)\,n_{m+4}$
with $\gcd(2^{2m-1}-1,2^{2m+7}-1)=2^{\gcd({2m-1,\,2m+7})}-1=1$, von Staudt--Clausen and Kummer's congruence
\cite[Ch.~15]{classical} give $v_p(g_{m})=0$ for every \emph{regular} $p$; thus the odd part of
$g_{m}$ is supported on \emph{irregular} primes dividing $n_{m}$
or $n_{m+4}$. Accordingly, $g=2^{e_m}$ if and only if no irregular
prime dividing $\gcd(S_{m},U_{m})$ divides $A_{m}$. This last statement, which holds for
$1\le m\le 40$, is the step we do not resolve for all $m$, owing to the arithmetic of
irregular primes, whose occurrence admits no closed description: an irregular prime $p$
may divide both $S_{m}$ and $U_{m}$ (for instance $p=5209$ at $m=322$: $5209\mid n_{322}$, so $5209\mid S_{322}$; and
$2^{651}\equiv 1\pmod{5209}$, so $5209\mid 2^{2m+7}-1$ and so $5209\mid U_{322}$), and
whether it also divides $A_{m}$ is governed by the residue of $c_m$ modulo $p$, which the
congruence~\eqref{eq:cm-congruence} does not control. Irregular primes lack a structural description that would exclude such coincidences
uniformly in $m$.
\begin{acknow}
The authors are sincerely grateful to Ramesh Kasilingam for his
valuable suggestions, comments, and probing questions, which have
improved the quality of this article. They also thank Wolfgang
L\"uck and Tibor Macko for answering several questions. The first author also
gratefully acknowledges financial support from the Prime Minister's
Research Fellowship, Government of India (PMRF ID:~2502403). 
\end{acknow}
\bibliographystyle{alpha}
\bibliography{reference1}
\end{document}